\documentclass[11pt,oneside,reqno]{amsart}
\usepackage{amsmath,amsfonts,latexsym,graphicx,amssymb,url}
\usepackage{hyperref}
\usepackage{amsmath,txfonts,pifont,bbding,pxfonts}
\usepackage{manfnt}
\usepackage{tabularx}
\usepackage{booktabs}
\usepackage{caption}

\usepackage[active]{srcltx}
\usepackage{cases}
\usepackage{wasysym,pstricks, enumerate,subfigure}
\usepackage{lscape,color}
\usepackage{soul}
\usepackage[none]{hyphenat}
\usepackage[mathscr]{euscript}
\usepackage{wrapfig, subfigure,graphicx}
\usepackage{enumitem}

\numberwithin{equation}{section}
\numberwithin{equation}{subsection}

\newtheorem{theorem}{Theorem}[section]
\newtheorem{corollary}[theorem]{Corollary}
\newtheorem{lemma}[theorem]{Lemma}
\newtheorem{proposition}[theorem]{Proposition}

\newtheorem{remark}[theorem]{Remark}

\numberwithin{equation}{section}
\numberwithin{equation}{subsection}

\begin{document}

\title[Inverse problems with hybrid lenses]{Inverse problems with hybrid lenses}
\author[I. Amro, F. Fneish, R. Kansoh,  A. Sabra, and W. Tabbara]{I. Amro, F. Fneish, R. Kansoh, A. Sabra, and W. Tabbara\\}
\providecommand{\keywords}[1]{\textbf{\textit{Index terms---}} #1}
\thanks{MSC Classification Codes: 34A55, 35F61, 35Q60, 78A05, 26B10.}
\thanks{
The authors were partially supported by the CAMS summer research program from the American University of Beirut. A. S. was partially supported by the University Research Board, Grant 104107 from the American University of Beirut.}
\address{Department of Mathematics, Center of Advanced Mathematical Sciences \\ American University of Beirut\\ Beirut}
\email{iza04@mail.aub.edu, faf21@mail.aub.edu, rak104@mail.aub.edu, asabra@aub.edu.lb, wkt00@mail.aub.edu}

\begin{abstract}
We design lenses composed of a combination of standard freeform refracting surface and flat metasurface refracting an arbitrary incident field into a collimated beam with a fixed direction. In the near-field case, we study the existence of such lenses refracting a bright object into a predefined image at the target.
\end{abstract}


\maketitle

\tableofcontents

\section{Introduction}
Snell's law of refraction \cite[Chapter 1]{born2013principles} describes the trajectory of a ray when passing from one medium to another. The interface separating both media creates a discontinuity in the refractive index causing a change in the direction of the propagating radiation. Due to recent technological advances, metasurfaces were introduced in optical designs. These are ultra-thin artificially engineered materials composed of nano-structures that introduce abrupt phase shifts along the optical path to bend light in non-standard ways \cite{achouri2021electromagnetic}. The corresponding generalized Snell's law was derived and verified in multiple papers in optical engineering for example \cite{doi:10.1126/science.1210713} and \cite{10.1007/978-3-030-50399-4_35} in two dimensions. The three-dimensional corresponding law is proved in \cite{GUTIERREZ2021102134} using the Fermat principle and in \cite{gutirrez2023maxwell} from the distributional Maxwell system.   %Metamaterial design is an accelerating field of research with promising applications in commercial industries such as smartphones camera CITE, antennas CITE, and medical imaging devices CITE. These optical systems include a combination of conventional surfaces and metasurfaces. 

In this paper, we study two inverse problems in Geometric Optics involving hybrid lenses composed of a standard refracting surface and a flat metasurface. We first study the existence of such a lens refracting a variable incident field of directions emitted from a planar source into a predefined constant direction, see Figure \ref{fig: general field}. We assume that the face closer to the source is a given conventional refracting surface and that the other face of the lens is the flat metasurface. The goal is to analyze the existence of a phase discontinuity in a neighborhood of every point on the metasurface so that all rays leave the lens along a given constant direction. To do this, we use the standard and the generalized Snell's laws to reduce the problem into a system of partial differential equations and show that the incident field must satisfy a curl condition \eqref{curl condition}. Assuming \eqref{curl condition}, we employ the implicit function theorem and the notion of envelop to find a sufficient condition on the lower face of the lens and the incident field so that the required phase discontinuity exists.
The considered model in this part is a non-imaging far-field inverse problem where we are not interested in creating an image but rather in the direction of the rays leaving the lens.

Using the far-field analysis, we next consider the following imaging problem in the near-field case. Given a bijective map $T$ between a planar source and a planar target, our goal is to find a hybrid lens that achieves $T$. In other words, every point in the source emits a ray that is refracted by the lens into its corresponding image defined by $T$. The rays entering and leaving the lens are assumed vertical see Figure \ref{Imaging Problem}. In this model, we show that the lower face of the lens satisfies a system of semilinear partial differential equations \eqref{eq:PDE of rho} and find in Theorem \ref{thm:Hartman verification} conditions on $T$ so that a solution exists. Once the lower face is found, the existence of the phase discontinuity will be deduced from the far-field analysis requiring additional conditions on the map $T$, see Theorem \ref{thm:non singularity}. We apply this analysis in Section \ref{subsec:Examples} to several examples in particular when $T$ is a magnification/contraction, and when $|T-I|$ is constant.x

We put our results in perspective both from the mathematical and optical points of view. The design of a two-dimensional convex, analytic standard lens focusing light rays emitted from one point source into a point image was first solved in \cite{Friedman:1987vg} using a fixed point type argument. The result was later generalized to three dimensions in \cite{Gutierrez:13} where the author constructed freeform lenses refracting rays emitted from a point source into a constant direction or a point image. The case of a general field was later studied in \cite{doi:10.1137/15M1030807} and necessary and sufficient conditions were found for the existence of $C^2$ lenses in $\mathbb R^3$ refracting an arbitrary incident field into a collimated beam. Reflective models and combinations of refracting and reflecting surfaces are studied in \cite{Gutierrez:14}. Illumination problems with one reflective or refractive surface are addressed in \cite{KochenginOlikerTempski}, \cite{Wang}, \cite{GUTIERREZ2014655}, \cite{KARAKHANYAN20201278}, \cite{10.4310/jdg/1279114301} for the point source case and in \cite{Doskolovich:18} and \cite{GutTou} for the planar source case. Illumination models involving a single lens (2 surfaces) are considered in \cite{Rolland:21}, and \cite{Gutierrez:2018wo}.

The standard refracting surfaces involved in the models of this paper are not assumed to be convex or concave. The use and design of such freeform surfaces became possible 
with the technological advances in manufacturing design, ultra-precision cutting, grinding, and polishing \cite{10.1117/12.2188523}, \cite{Rolland:21}. Their unconstrained geometry presents a unique opportunity to achieve optical tasks that are not attainable with traditional concave/convex surfaces \cite{Thompson:12}, \cite{HicksMirror}. Though revolutionary, freeform surfaces have their limitations as their potentially complicated geometry makes them costly to manufacture \cite{inproceedings} and could result in ray obstruction \cite{Gutierrez:2018wo}. 
Recent advances in nanotechnology have further expanded the possibilities to include Metamaterial in optics. This is an accelerating field of research with promising applications in commercial industries such as smartphone cameras \cite{doi:10.1126/science.aam8100}, antennas \cite{canet2019metamaterials}, and medical imaging devices \cite{9201430}. Their applicability gained public interest for their efficiency in eliminating chromatic aberration \cite{Aieta:13},  \cite{CapassoVisible}. Mathematically, such claims were investigated in \cite{GUTIERREZ2021102134}, where the authors also studied the existence of single-element metasurfaces refracting collimated incident field into an arbitrary field direction. Corresponding illumination problems and connections with optimal transport are explored in \cite{Gutierrez:18}, \cite{gutierrez2022metasurfaces}. 
Hybrid combination of standard lenses and metasurfaces also appears in the literature with several applications to image corrections \cite{HybridCapasso}, \cite{Shih:22}. 

Within this evolving landscape of optical physics, our paper takes root. We develop mathematical tools to study two inverse problems in geometric optics involving freeform refractors and metasurfaces. The paper is organized as follows. Section \ref{sec:Snell} introduces the standard and the generalized Snell's laws. In Section \ref{sec: uniform refraction}, we precisely state and solve the far-field problem while finding necessary conditions on the incident field in Section \ref{subs: Necessary condition }, and sufficient conditions on the incident field and lower face of the hybrid lens in Section \ref{subsec:sufficient condition} for a solution to exist. Section \ref{subsec:vertical} considers the case when 
the field entering and leaving the lens is vertical and formulates the corresponding sufficient conditions in Theorem \ref{thm:sufficient for vertical field}. In Section \ref{sec:Imaging}, we study the near-field imaging problem and formulate the corresponding system of PDEs in Proposition \ref{prop: PDE of rho} that must be satisfied by the lower face of the desired lens. We find in Theorem \ref{thm:Hartman verification} necessary and sufficient conditions for the existence of local solution, and then apply in Section  \ref{sec: existence of phi} the analysis of Section \ref{subsec:sufficient condition} to find additional conditions on the map $T$ so that a phase discontinuity exists in a neighborhood of the metasurface. Our paper concludes with a few examples of allowable maps $T$ in three dimensions, Section \ref{subsec:Examples}, and in two dimensions, Section \ref{subsec:2D}.
\paragraph{\bf List of Notations.}
Before embarking on our analysis we introduce some notation that will be used throughout the paper.
\begin{itemize}
\item All vectors in $\mathbb R^n$ are assumed to be row vectors.
%\item $A^t$ is the transpose of a matrix $A$.
\item If $A$ is an $m\times n$ matrix and $B$ is an $m\times p$ matrix then $A\otimes B=A^tB$.
\item For ${\bf a}=(a_1,a_2)$, ${\bf a}_{\perp}=(-a_2,a_1)$. 
\item For ${\bf F}(x)=(F_1(x),\cdots , F_m(x))$ a $C^1$ field in a domain in $\Omega \subseteq \mathbb R^n$, the derivative $D{\bf F}(x)$ is the $m\times n$ matrix $$D{\bf F}(x)=\left(\dfrac{\partial F_k}{\partial x_i}\right)_{1\leq k\leq m, 1\leq i\leq n}.$$
\item The scalar curl of a two dimensional $C^1$ vector field ${\bf a}(x)=(a_1(x),a_2(x))$ is 
\[
\nabla\times {\bf a}(x)=(a_2(x))_{x_1}-(a_1(x))_{x_2}
\]
\item The scalar cross product of two dimensional vectors ${\bf a}=(a_1,a_2)$ and ${\bf b}=(b_1,b_2)$ is
$${\bf a}\times {\bf b}=\left|\begin{matrix} a_1 & a_2\\ b_1 & b_2\end{matrix}\right|.$$
\item The cross product of three dimensional vectors ${\bf v}=(v_1,v_2,v_3)$ and ${\bf v'}=(v_1',v_2',v_3')$ is
$${\bf v}\times {\bf v'}=\left|\begin{matrix} {\bf i} & {\bf j} & {\bf k}\\ v_1 & v_2 & v_3\\ v_1' & v_2' & v_3'\end{matrix}\right|.$$
\end{itemize}

\section{Preliminaries: Snell's laws of refraction}\label{sec:Snell}
The refractive index of a material corresponding to an electromagnetic wave with angular frequency $\omega$ is given by $n=\dfrac{c}{v}$ with $c$ the speed of the wave in a vacuum and $v$ its apparent velocity in the medium. The wave number $k=\dfrac{\omega}{v}$ is defined as the number of wave cycles per unit distance in the medium. Denoting by $k_0=\dfrac{\omega}{c}$ the wave number in vacuum, we get  
 that $n=\dfrac{k}{k_0}.$
The refractive index $n$ depends on both the medium and the propagating wave, see  \cite[Chapter 3]{HechtOptics} for a more in-depth interpretation of the involved parameters.

This section provides a brief review of the standard and the generalized Snell's laws and lays down the primary formulations needed for solving the inverse problems posed in this paper.
\subsection{The standard Snell's law}\label{sec:Standard Snell}
Let $\Gamma$ be a $C^1$ surface in $\mathbb R^3$ separating two homogeneous and isotropic media I and II. Upon striking $\Gamma$ at a point $P$, a light wave with angular frequency $\omega$ propagating along the unit direction ${\bf x}$ in I refracts along the unit direction ${\bf m}$ in II abiding the Snell's law of refraction
\begin{equation}\label{eq:vec Snell}
n_1 ({\bf x}\times \nu ) = n_2({\bf m} \times \nu) .\end{equation} $n_1$ and $n_2$ are respectively the refractive indices of media I and II corresponding to $\omega$, and $\nu$ the unit normal to $\Gamma$ at $P$ toward medium II, that is $x \cdot \nu \ge 0$.

%The stated law implies that $x,m,\nu$  lie in the same plane passing through $P$, called the plane plane of incidence. 

Setting $\kappa = \dfrac{n_2}{n_1}$, \eqref{eq:vec Snell} yields the existence of $\lambda\in \mathbb R$ so that
\begin{equation} 
\label{eq: standard Snell's law}
     {\bf x} - \kappa {\bf m} = \lambda \nu.
\end{equation} 
In fact, from the calculations in \cite{GutHuanNear}
\begin{equation}\label{eq: lambda}
\lambda={\bf x}\cdot\nu-\sqrt{\kappa^2-1+({\bf x}\cdot\nu)^{2}}=\dfrac{1-\kappa^2}{{\bf x}\cdot\nu+\sqrt{\kappa^2-1+({\bf x}\cdot\nu)^2}}.
\end{equation}

%if we denote by $\theta_1$ the angle in the plane of incidence between  between $x$ and $\nu$, and $\theta_2$ the angle between $m$ and $\nu$, we obtain the infamous scalar Snell's law
%\[
%\sin \theta_1 = \kappa \sin \theta_2.
%\]

Notice the following
\begin{itemize}
    \item  If $\kappa>1$, the term under the square root in \eqref{eq: lambda} is always positive, and $\lambda<0$. In this case, refraction into medium \text{II} occurs for all incident directions. 
    \item If $\kappa<1$, refraction occurs if and only if ${\bf x}\cdot \nu\geq \sqrt{1-\kappa^2}$, and for such incident directions $\lambda>0$.
\end{itemize}

The dot product ${\bf x\cdot m}$ corresponds to the deviation between the incident and the refracted directions. Dotting \eqref{eq: standard Snell's law} with ${\bf x}$, and using \eqref{eq: lambda} results in 
$${\bf x}\cdot {\bf m}=\dfrac{1}{\kappa}\left(1-\lambda \, {\bf x}\cdot \nu\right)=\dfrac{1}{\kappa}\left(1-\dfrac{(1-\kappa^2)({\bf x}\cdot \nu)}{{\bf x}\cdot \nu+\sqrt{\kappa^2-1+({\bf x}\cdot \nu)^2}}\right)=\dfrac{1}{\kappa}\left(1-(1-\kappa^2)\psi({\bf x}\cdot \nu)\right),$$
with $\psi(t)=\dfrac{t}{t+\sqrt{\kappa^2-1+t^2}}.$
Hence
\begin{itemize}
    \item For $\kappa>1$, $\psi$ is increasing for $t\in [0,1]$, so 
\begin{equation*}%\label{cond: kappa>1}
{\bf x}\cdot {\bf m}\geq \dfrac{1}{\kappa}.
\end{equation*}
\item For $\kappa<1$, $\psi$ is decreasing for $t\in [\sqrt{1-\kappa^2},1],$ so
\begin{equation*}%\label{cond:kappa<1}
{\bf x}\cdot {\bf m}\geq \kappa.
\end{equation*}
\end{itemize}

\subsection{The generalized Snell's law}\label{sec:Generalized Snell}
Denote by $(\Gamma,\phi)$ the metasurface with $\Gamma$ a $C^1$ surface separating media I and II and $\phi$ a $C^1$ function representing the phase discontinuity defined on a neighborhood of every point of $\Gamma$. 
An electromagnetic wave with frequency $\omega$
propagating in medium 
I with unit 
direction ${\bf x}$ is 
refracted by the 
metasurface $(\Gamma,\phi)$ at the point of incidence $P$
 along 
the unit direction ${\bf m}$ 
into II 
according to the 
generalized 
Snell's law of refraction
\begin{equation}\label{eq: generalized Snell law}
    n_1\left({\bf x} -\frac{1}{k}\nabla \phi\right) \times \nu = n_2{\bf m}\times \nu.
    \end{equation}
$n_1$ and $n_2$ are respectively the refractive indices of media I and II corresponding to $\omega$, $\nabla \phi$ denotes the gradient of $\phi$ at $P$, $k$ is the wave number corresponding to $\omega$ in medium I, and $\nu$ the unit normal at $P$ toward medium II. In this case, media I and II could be identical.

Letting $\kappa=\frac{n_2}{n_1}$, then from \eqref{eq: generalized Snell law} 
\begin{equation}
\label{eq:vectorial generalized Snell's law}
    {\bf x}-\frac{1}{k}\nabla \phi - \kappa {\bf m} = \mu \nu
\end{equation}
 with, from \cite{GUTIERREZ2021102134},
  \begin{equation} \label{equation of mu} \mu = \left({\bf x} - \frac{1}{k}\nabla\phi\right) \cdot \nu - \sqrt{\kappa^2- \left|{\bf x} -  \frac{1}{k}\nabla\phi\right|^2+\left[\left({\bf x} - \frac{1}{k}\nabla\phi\right) \cdot \nu\right]^2  }. \end{equation}
For refraction to occur, it is required that
\begin{equation}\label{eq:condition generalized}
\left[\left({\bf x} - \frac{1}{k}\nabla\phi\right) \cdot \nu\right]^2\geq \left|{\bf x}-\frac{1}{k}\nabla \phi\right|^2-\kappa^2.
\end{equation}

It's conventional in metalens design to have a tangential 
phase discontinuity \cite{LeeTangential}, 
that is, $\nabla \phi\cdot \nu=0$ at the point of incidence $P$. In this case,
\begin{equation}\label{eq: mu tangential}
\mu={\bf x}\cdot \nu-\sqrt{\kappa^2-\left|{\bf x}-\frac{1}{k}\nabla \phi\right|^2+({\bf x}\cdot \nu)^2},
\end{equation}
and the condition for refraction \eqref{eq:condition generalized} becomes
\begin{equation*}%\label{eq: condition tangential}
({\bf x}\cdot \nu)^2\geq \left|{\bf x}-\frac{1}{k}\nabla \phi\right|^2-\kappa^2.
\end{equation*}

Notice that in the case of the absence of phase discontinuity, $\nabla \phi=0$, we recover the formulae for standard refracting surfaces obtained in Section \ref{sec:Standard Snell}.
 
 %\\Note that the refracted
% vector $m$ is not on the plane generated by $x$ and $\nu$ anymore.
% \\In the case where $\phi$ is constant, we recover the classical Snell's law.
% \\If $\nabla \phi$ is tangential to $\Gamma$, $\nabla \phi \cdot \nu = 0$ and (2.2.2) becomes $(n_1 x - n_2 m)\cdot\nu = \mu$.

\section{Uniform Refraction of a general incident field}\label{sec: uniform refraction}
%Let $\Omega$ $\subset \R^2$ be a source emitting a  $C^2$ general field of monochromatic light rays of unit direction $e(x) = (e_1(x), e_2(x), e_3(x))$. Let $\sigma_1:\big\{(U(x), \Theta(U(x)) , x \in \Omega)\big\}$ be a surface obeying the standard Snell's law (2.1.2) and $\sigma_2: \big\{(Q(x), a), x \in \Omega \big\}$ be a surface with a $C^1$ function $\phi (Q(x), a)$ defined in a neighborhood of each $(Q(x),a) \in \sigma_2$ obeying the generalized Snell's law (2.2.2). So that $\forall x \in \Omega$, each ray emitted from $x \in \Omega$ gets refracted by $\sigma_1$ along the unit direction $m(x)$. We discuss the existence of the function $\phi (Q(x), a)$ that would allow the refraction of $m(x)$ by $\sigma_2$ along the constant unit direction $w = (0,0,1)$. We denote the refractive index of the medium below $\sigma_1$ by $n_1$, that of the medium sandwiched between $\sigma_1$ and $\sigma_2$ by $n_2$, and that above $\sigma_2$ by $n_3$.
%\\OR
\subsection{Problem setup}\label{subsec:Setup}
We are given an open and connected domain $\Omega\subseteq \mathbb R^2$, a $C^1$ unit vector field ${\bf e}(x)=(e_1(x),e_2(x),e_3(x)):=(e'(x),e_3(x))$ defined on $\Omega$, with $e'(x)=(e_1(x),e_2(x))$ and $e_3(x)>0$, and a $C^2$ conventional refracting surface $\sigma_1$ above the horizontal plane $\{x_3=0\}$, and below the plane $\{x_3=a\}$ with $a>0$. Denote by $n_1$ the refractive index of medium I below $\sigma_1$ and $n_2$ the refractive index of medium II between $\sigma_1$ and $\{x_3=a\}$.

 Monochromatic radiation with frequency $\omega$ are issued from $(x,0)$, $x=(x_1,x_2)\in \Omega$, with direction ${\bf e}(x)$ and strike $\sigma_1$ at the point $P(x)$. Let $\rho(x)=\left|P(x)-(x,0)\right|$ be the length of the trajectory traversed by the ray with direction ${\bf e}(x)$ in medium I. Assume that medium II is denser than medium I, i.e. $n_2>n_1$. Letting $\kappa_1=\dfrac{n_2}{n_1}$, $\kappa_1>1$, then from Subsection \ref{sec:Standard Snell}, refraction occurs at $P(x)$. The refracted ray propagates into medium II along the unit direction ${\bf m}(x)$ given by the Snell's law \eqref{eq: standard Snell's law}
\begin{equation}\label{eq: m inside II}
{\bf m}(x)=\dfrac{1}{\kappa_1}({\bf e}(x)-\lambda \nu):=(m_1(x),m_2(x),m_3(x))
\end{equation}
with $\nu$ the unit normal to $\sigma_1$ at $P(x)$ toward medium II, and 
$\lambda={\bf e}(x)\cdot \nu-\sqrt{\kappa_1^2-1+({\bf e}(x)\cdot \nu)^2}.$
Since $\kappa_1>1$, then from Subsection \ref{sec:Standard Snell}, $\lambda<0$. We assume that $\nu=(\nu_1,\nu_2,\nu_3)$ with $\nu_3>0$, then since $e_3(x)>0$ it follows from \eqref{eq: m inside II} that $m_3(x)>0$.
The refracted ray with direction ${\bf m}(x)$ strikes a flat horizontal metasurface $(\sigma_2,\phi)$, with $\sigma_2\subset\{x_3=a\}$, at the point $(Q(x),a):=(Q_1(x),Q_2(x),a)$. Let $d(x)=|(Q(x),a)-P(x)|$ be the length of the trajectory traversed by the ray with direction ${\bf m}(x)$ in medium II. Having that ${\bf m}$ is unit we parametrize $\sigma_2$ as follows
\begin{equation}\label{eq:parametrization of sigma2}
(Q(x),a)= P(x)+d(x){\bf m}(x)=(x,0)+\rho(x){\bf e}(x)+d(x){\bf m}(x).
\end{equation}
Equating the vertical components in \eqref{eq:parametrization of sigma2} yields the following formula for $d$
\begin{equation}\label{eq: formula for d general}
d(x)=\dfrac{a-\rho(x)e_3(x)}{m_3(x)}.
\end{equation}
Let $n_3$ be the refractive index of the medium above the plane $\{x_3=a\}$, referred to as medium III.
\begin{figure}[h]
\includegraphics[scale=0.6]{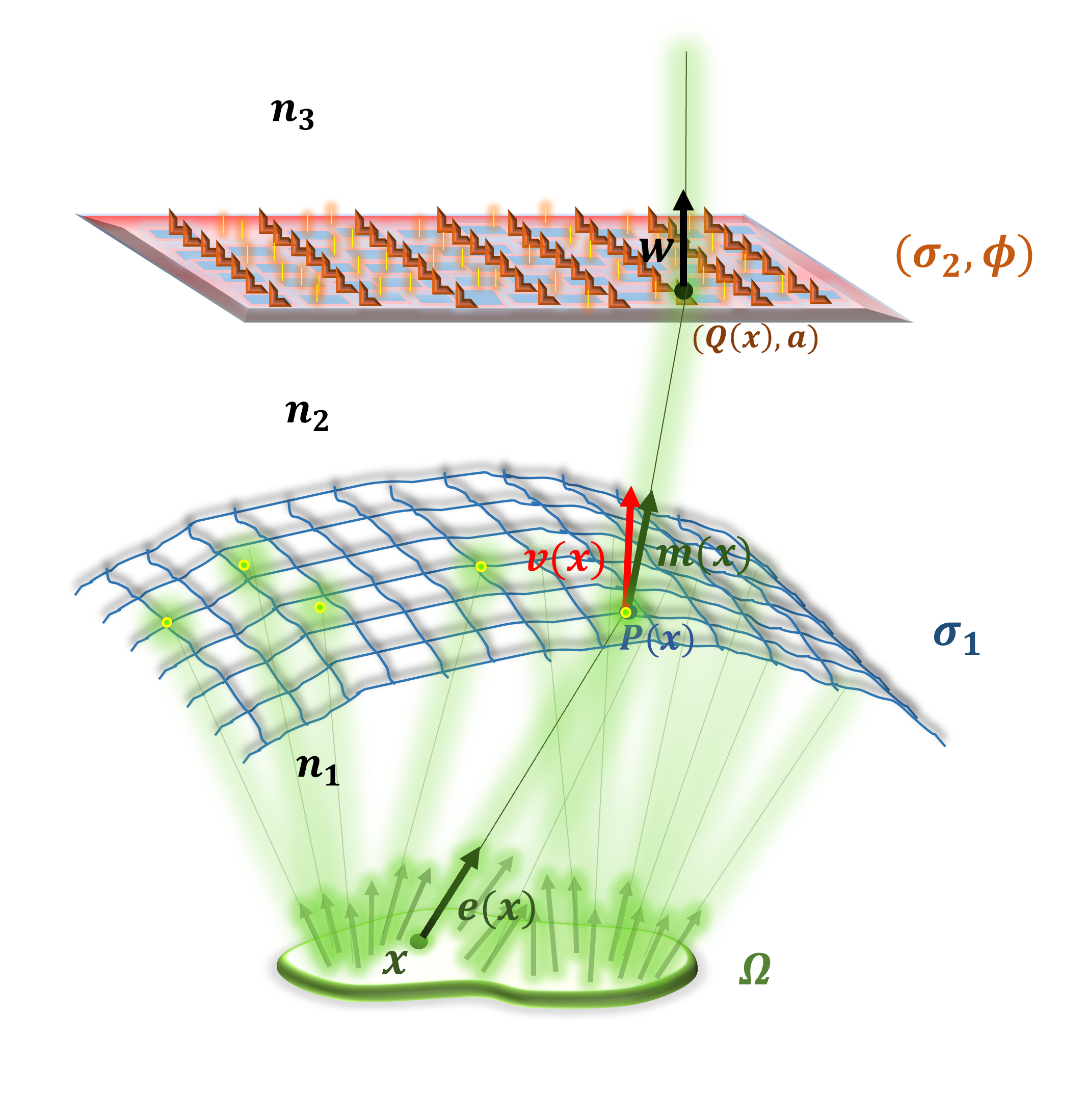}
\caption{}
\label{fig: general field}
\end{figure}

With the above setting, see Figure \ref{fig: general field}, our goal is to study the existence of a $C^1$ phase discontinuity $\phi:=\phi(u_1,u_2,u_3)$ defined on a neighborhood of every point of $\sigma_2$, so that at $(Q(x), a)$, the ray with direction ${\bf m}(x)$ is refracted into medium III along the vertical direction ${\bf w}=(0,0,1).$
We denote by $(\sigma_1,(\sigma_2,\phi))$ the lens with lower face the standard refracting surface $\sigma_1$ and upper face the planar metasurface $(\sigma_2,\phi)$.

\subsection{Necessary condition} \label{subs: Necessary condition } We start by proving a necessary condition for the existence of a solution to the problem introduced in Section \ref{subsec:Setup}.

\begin{proposition}\label{prop:necessary}
If a lens $(\sigma_1,(\sigma_2,\phi))$ refracting the $C^1$ field ${\bf e}(x)=(e'(x),e_3(x))$ into ${\bf w} = (0,0,1)$ exists then $e'=\nabla h$ for some $C^2$ function $h$, and so 
\begin{equation}\label{curl condition}
\nabla \times e'(x)=0
\end{equation}
%equivalently $e'=\nabla h$ for some $C^2$ function $h$.
\end{proposition}

\begin{proof}
Suppose that there exists a phase discontinuity $\phi:=\phi(u_1,u_2,u_3)$ defined on a neighborhood of every point of $\sigma_2$ so that the lens $(\sigma_1,(\sigma_2,\phi))$ refracts all the rays of direction ${\bf m}(x)$ given by \eqref{eq: m inside II} into the direction ${\bf w}.$ The generalized Snell's law \eqref{eq:vectorial generalized Snell's law} applied to the incident direction ${\bf m}(x)$, the refracted direction ${\bf w}=(0,0,1)$, and the normal $(0,0,1)$ to the horizontal metasurface $\sigma_2$ implies that
$${\bf m}(x)-\frac{1}{k}\nabla \phi(Q(x),a)-\kappa_2(0,0,1)=\mu(0,0,1)$$
with $\kappa_2=\dfrac{n_3}{n_2}$ and from \eqref{equation of mu} 
$$\mu=m_3(x)-\dfrac{1}{k}\phi_{u_3}(Q(x),a)-\sqrt{\kappa_2^2-\left|{\bf m}(x)-\dfrac{1}{k}\nabla \phi(Q(x),a)\right|^2+\left[m_3(x)-\frac{1}{k}\phi_{u_3}(Q(x),a)\right]^2}.$$ This is equivalent to the following system
\begin{equation}\label{m1,m2}
m_i(x)=\dfrac{1}{k}\phi_{u_i}(Q(x),a),\quad i=1,2.
\end{equation}

Let $f(x)=\frac{1}{k}\phi(Q(x),a)$, from the chain rule and \eqref{m1,m2} 
\begin{equation*} %\label{f_xi equ} 
f_{x_i}(x)=\frac{1}{k}(Q_1)_{x_i}\phi_{u_1}(Q(x),a)+\frac{1}{k}(Q_2)_{x_i}\phi_{u_2}(Q(x),a)=(Q_{x_i},0)\cdot {\bf m} \end{equation*}
Recall that $|{\bf m}(x)|=1$. Differentiating \eqref{eq:parametrization of sigma2} with respect to $x_i$ and dotting with $\bf m$ we obtain
$$f_{x_i}(x)=
P_{x_i}(x)\cdot {\bf m}(x)+d_{x_i}(x).$$
From Snell's law \eqref{eq: standard Snell's law} at $P(x)=(x,0)+\rho(x){\bf e}(x)$, the vector ${\bf e}(x)-\kappa_1 {\bf m}(x)$ is parallel to the normal at $P$, and then since $|{\bf e}(x)|=1$
\begin{equation}\label{eq: m dot Pxi}
P_{x_i}(x)\cdot {\bf m}(x)=\dfrac{1}{\kappa_1}P_{x_i}(x)\cdot {\bf e}(x)=\dfrac{1}{\kappa_1}e_i(x)+\dfrac{1}{\kappa_1}\rho_{x_i}(x).
\end{equation}
We conclude that
$$f_{x_i}=\dfrac{1}{\kappa_1}e_{i}+\dfrac{1}{\kappa_1}\rho_{x_i}+d_{x_i}.$$
Hence $$e_i=\left(\kappa_1f-\rho-\kappa_1d\right)_{x_i}\qquad i=1,2.$$
Since $e_i\in C^1$, then $\kappa_1f-\rho-\kappa_1d\in C^2$. By the mixed derivative theorem, we obtain \eqref{curl condition}. %\begin{equation}\label{curl condition}\nabla
%$\nabla\times  e'(x)=0$%\end{equation} 
%which implies since $\Omega$ is simply connected that there is exists scalar function $h:=h(x_1,x_2)$ such that $e'(x)=\nabla h$.
\end{proof}
%\end{equation}

\begin{remark}\label{examples}\rm The necessary condition of Proposition \ref{prop:necessary} is satisfied for collimated incident fields, that is, ${\bf e}(x)=(e_1,e_2,e_3)$ is constant.

The curl condition is also satisfied when rays are emitted from a point source $R$ toward the surface $\sigma_1$ above $R$. In this case, we view $\Omega$ as the intersection of the rays with a virtual plane between $R$ and $\sigma_1$. With an appropriate choice of coordinates, the incident rays can be described by the field ${\bf e}(x)=\dfrac{(x,0)-R}{|(x,0)-R|}$. Here, $e'(x)=\nabla|(x,0)-R|$.
\end{remark}

\begin{remark}\rm
If $\phi$ solves \eqref{m1,m2}, then the internal reflection condition \eqref{eq:condition generalized} follows. In fact, for such $\phi$
$$\left|{\bf m}(x)-\dfrac{1}{k}\nabla\phi\right|^2-\kappa_2^2<\left|{\bf m}(x)-\dfrac{1}{k}\nabla\phi\right|^2=\left(m_3-\dfrac{1}{k}\phi_{u_3}\right)^2=\left[\left({\bf m}(x)-\frac{1}{k}\nabla \phi\right)\cdot (0,0,1)\right]^2.$$
\end{remark}

\subsection{Sufficient condition.}\label{subsec:sufficient condition} Given an incident field 
satisfying the necessary condition in Proposition \ref{prop:necessary}, and assuming that ${\bf e}$ and ${\bf m}$ are $C^2$, we use the implicit function theorem to find a sufficient condition for the existence of $\phi$ solving the system \eqref{m1,m2}. %Our analysis requires higher order smoothness for the fields ${\bf e}$ and ${\bf m}$.
\begin{theorem}\label{thm:sufficient}
Suppose that the unit field ${\bf e}(x)=(e'(x),e_3(x))$ is in $C^2(\Omega)$ with $e'=\nabla h$ for some function $h$, and $e_3>0$. Assume, moreover, that ${\bf m}(x)$ given in \eqref{eq: m inside II} is in $C^2(\Omega)$. If 
{\footnotesize
\begin{equation}\label{eq:big det}
\det\left(D^2h+(1-\kappa_1{\bf e}\cdot {\bf m})D^2\rho-\kappa_1(D\rho\otimes ({\bf m}D{\bf e})+({\bf m}D{\bf e})\otimes D{\bf \rho})-\kappa_1\rho(D({\bf m}D{\bf e})-D{\bf e}\otimes D{\bf m})+\kappa_1 d D{\bf m}\otimes D{\bf m}\right)\neq 0
\end{equation}
}
at $x_0\in \Omega$ then there exists a neighborhood $U\subseteq \Omega$ of $x_0$ and a $C^1$ tangential phase discontinuity $\phi$ defined in a neighborhood of $(Q(x_0),a)\in \sigma_2$ such that for every $x\in U$, the field ${\bf e}(x)$ is refracted by the lens $(\sigma_1,(\sigma_2,\phi))$ into the vertical direction ${\bf w}=(0,0,1).$
\end{theorem}

\begin{proof}
%Our goal is to find conditions on $\rho$ and ${\bf e}$ so that for every $x\in \Omega$ there exists a $C^1$ function $\phi:=\phi(u_1,u_2,u_2)$ defined on a neighborhood of $(Q(x),a)$ that solves the system \ref{m1,m2}. We also require $\phi$ to be tangential to $\sigma_2\subseteq \{u_3=a\}$, this is saying $\phi_{u_3}(Q(x),a)=0$.

Since $\rho$ and ${\bf e}$ are in $C^2(\Omega)$ then so is $P$. Also, from \eqref{eq: formula for d general} and the facts that ${\bf m}\in C^2$ and $m_3>0$ it follows that $d\in C^2(\Omega)$, and therefore $Q\in C^2(\Omega)$.

From \eqref{eq: m dot Pxi}, for $i=1,2$
$$(Q_{x_i}(x),0)\cdot {\bf m}(x)= (Q(x),a)_{x_i}\cdot{\bf m}(x)=P_{x_i}(x)\cdot {\bf m}(x)+d_{x_i}(x)=\dfrac{1}{\kappa_1}h_{x_i}(x)+\dfrac{1}{\kappa_1}\rho_{x_i}(x)+d_{x_i}(x).$$
Define the following maps in $C^2(\Omega)$
$$ f(x)=\dfrac{1}{\kappa_1}h(x)+\dfrac{1}{\kappa_1}\rho(x)+d(x),\qquad H(x)=(Q(x),f(x)), \qquad N(x)=(m_1(x),m_2(x),-1).$$
Notice that for $i=1,2$
\begin{equation}\label{eq:useful}
    H_{x_i}\cdot N=Q_{x_i}\cdot (m_1,m_2)-f_{x_i}=(Q_{x_i},0)\cdot {\bf m}-f_{x_i}=0.
\end{equation}

For $u=(u_1,u_2,u_3)\in \mathbb R^3$, and $x\in \Omega$, define the $C^2$ function
$F(u,x)=(u-H(x))\cdot N(x),$
and the $C^1$ map $G: \mathbb R^3\times \Omega\mapsto \mathbb R^3$ given by $G(u,x)=(F(u,x),F_{x_1}(u,x),F_{x_2}(u,x)).$ We consider the system
\begin{equation}\label{eq:system for G}
G(u,x)=(0,0,0).
\end{equation}
Fix $x_0\in \Omega$. From the formula of $F$ and \eqref{eq:useful}, $(H(x_0),x_0)=(Q(x_0),f(x_0),x_0)$ solves \eqref{eq:system for G}. Assume that
\begin{equation}\label{eq:non-zero Jacobian}
\det\left(\dfrac{\partial G}{\partial(u_3,x_1,x_2)}(H(x_0),x_0)\right)\neq 0,
\end{equation} 
then by the implicit function theorem there exists an open neighborhood $O\subseteq \mathbb R^2$ of $Q(x_0)$, an open neighborhood $W\subseteq \mathbb R\times \Omega$ of $(f(x_0),x_0)$, and unique functions $g_1,g_2,g_3\in C^1(O)$ such that for $(u_1,u_2)\in O$
\begin{equation}\label{eq:system in U}
G(u_1,u_2, g_1(u_1,u_2),g_2(u_1,u_2),g_3(u_1,u_2))=(0,0,0).
\end{equation}

Since $f$ and $Q$ are continuous then there exists a neighborhood $U\subseteq \Omega$ of $x_0$ such that for each $x$ in that neighborhood $(f(x),x)\in W$, and $Q(x)\in O$. 
For such $x$, $(H(x),x)=(Q(x),f(x),x)$ solves \eqref{eq:system for G}. Therefore, by the uniqueness of $g_1,g_2,$ and $g_3$,  it follows that for every $x\in U$, 
\begin{equation}\label{eq: g1g2g3}
g_1(Q(x))=f(x), g_2(Q(x))=x_1, g_3(Q(x))=x_2.
\end{equation}

We prove that the function $\phi: O\times \mathbb R\mapsto \mathbb R$ defined as follows
$$\phi(u_1,u_2,u_3)=k g_1(u_1,u_2),$$
is our desired phase discontinuity. It's obvious that $\phi$ is tangential to the metasurface since $\nabla \phi\cdot (0,0,1)=k(g_1)_{u_3}=0$. It remains to show that $\phi$ solves the system \eqref{m1,m2} for every $x\in U$. From the formula of $G$ and \eqref{eq:system in U},  
$F(u_1,u_2,g_1(u_1,u_2),g_2(u_1,u_2),g_3(u_1,u_2))=0$, for every $(u_1,u_2)\in O$. Differentiating with respect to $u_i$, $i=1,2$, and using \eqref{eq:system in U} again, yields
$$0=F_{u_i}+F_{u_3} (g_1)_{u_i}+F_{x_1}(g_2)_{u_i}+F_{x_2}(g_3)_{u_i}=F_{u_i}+F_{u_3}(g_1)_{u_i}.$$
Particularly, for $(u_1,u_2)=Q(x)$, $x\in U$, from \eqref{eq: g1g2g3}
$$0=F_{u_i}(Q(x),f(x),x)+F_{u_3}(Q(x),f(x),x)\,(g_1)_{u_i}(Q(x))=m_i(x)-\frac{1}{k}\phi_{u_i}(Q(x),a),$$
%We have that for $x\in U$, $Q(x)\in O$, then for $(u_1,u_2)=Q(x)$ and $(Q(x),f(x),x)=(Q(x),g_1(Q(x)),g_2(Q(x)),g_3(Q(x)))$, and then the above identity holds when evaluated at $(Q(x),f(x),x)$.
%Notice that $F_{u_i}(Q(x),f(x),x)=N_i(x)=$ and $F_{u_3}(Q(x),f(x),x)=N_3(x)=-1$, therefore
%$$0=m_i(x)-(g_1)_{u_i}(Q(x))=m_i(x)-\dfrac{1}{k}\phi_{u_i}(Q(x),a),$$
concluding that $\phi$ satisfies the system \eqref{m1,m2}.

% \eqref{m1,m2}, also implies that
%$$\left|{\bf m}(x)-\frac{1}{k}\nabla \phi(x)\right|^2 =m_3(x)^2>({\bf m}(x)\cdot (0,0,1))^2-\kappa^2,$$
%and so the refraction condition \eqref{eq:eq: condition tangential} is verified, and all rays ${\bf m}(x)$ are refracted at $(Q(x),a)$ into the vertical direction ${\bf w}.$

It remains to simplify the expression of the Jacobian determinant in \eqref{eq:non-zero Jacobian}. From \eqref{eq:useful},
$$\det\left(\dfrac{\partial G}{\partial (u_3,x_1,x_2)}(H(x_0),x_0)\right)=\begin{vmatrix} F_{u_3} & F_{x_1} & F_{x_2}\\ F_{x_1u_3} & F_{x_1x_1} & F_{x_1x_2} \\ F_{x_2 u_3} & F_{x_2x_1} & F_{x_2x_2}\end{vmatrix}(H(x_0),x_0)=\begin{vmatrix} -1 & 0 & 0 \\ 0 & -H_{x_1}\cdot N_{x_1} & -H_{x_2}\cdot N_{x_1} \\ 0 & -H_{x_1}\cdot N_{x_2} & -H_{x_2}\cdot N_{x_2}\end{vmatrix}(x_0).$$
Hence \eqref{eq:non-zero Jacobian} is equivalent to 
$\det( (H_{x_i}\cdot N_{x_j})_{i,j=1,2})\neq 0$ at $x_0$.

For $i,j=1,2$
\[    H_{x_i}\cdot N_{x_j}=(Q_{x_i},f_{x_i})\cdot( (m_1)_{x_j}, (m_2)_{x_j}, 0)\\
    =(Q,a)_{x_i}\cdot {\bf m}_{x_j}\\
    =(P+d\,{\bf m})_{x_i}\cdot {\bf m}_{x_j}\\
    = P_{x_i}\cdot {\bf m}_{x_j}+d {\bf m}_{x_i}\cdot {\bf m}_{x_j}.\]
From \eqref{eq: m dot Pxi}, and the fact that $e'=\nabla h$
\begin{align*}
    P_{x_i}\cdot {\bf m}_{x_j}=(P_{x_i}\cdot {\bf m})_{x_j}-P_{x_ix_j}\cdot {\bf m}&=\dfrac{1}{\kappa_1}h_{x_ix_j}+\dfrac{1}{\kappa_1}\rho_{x_ix_j}-(\rho{\bf e})_{x_ix_j}\cdot {\bf m}\\
    &=\dfrac{1}{\kappa_1}h_{x_i x_j} +\left(\dfrac{1}{\kappa_1}-{\bf e}\cdot {\bf m}\right) \rho_{x_ix_j}-(\rho_{x_i}{\bf e}_{x_j}+\rho_{x_j}{\bf e}_{x_i})\cdot {\bf m}-\rho (({\bf e}_{x_i}\cdot {\bf m})_{x_j}-{\bf e}_{x_i}\cdot {\bf m}_{x_j})
\end{align*}
Concluding that 
$$
(H_{x_i}\cdot N_{x_j})_{i,j=1,2}=\dfrac{1}{\kappa_1}D^2h+\left(\dfrac{1}{\kappa_1}-{\bf e}\cdot {\bf m}\right)D^2\rho-(D\rho\otimes ({\bf m}D{\bf e})+({\bf m}D{\bf e})\otimes D{\bf \rho})-\rho(D({\bf m}D{\bf e})-D{\bf e}\otimes D{\bf m})+d D{\bf m}\otimes D{\bf m}.$$
  Multiplying above by $\kappa_1$, we obtain the sufficient condition \eqref{eq:big det}.
\end{proof}

\begin{remark}\label{rmk:collimated}\rm
 Assume that the field ${\bf e}=(e_1,e_2,e_3)$ is constant, with $e_3>0$. From Remark \ref{examples}, the necessary condition is satisfied. The sufficient condition \eqref{eq:big det}
reduces to 
\begin{equation}\label{eq:det for collimated}
\det \left((1-\kappa_1{\bf e}\cdot {\bf m})D^2\rho+\kappa_1 d\, D{\bf m}\otimes D{\bf m}\right)\neq 0
\end{equation}
at $x_0\in \Omega$.

In this case, if ${\bf e}\cdot {\bf \nu}>0$ and $D^2\rho$ is negative definite then \eqref{eq:det for collimated} follows at every point. In fact, from Section \ref{sec:Standard Snell}, since $\kappa_1>1$ then ${\bf e}\cdot {\bf m}\geq \dfrac{1}{\kappa_1}$ with equality if and only if ${\bf e}\cdot {\nu}=0$. Therefore if ${\bf e}\cdot \nu>0$ and $D^2\rho<0$, then the matrix $(1-\kappa_1{\bf e}\cdot{\bf m})D^2\rho$ is positive definite. On the other hand, $d>0$ and $D{\bf m}\otimes D{\bf m}$ is positive semi-definite, hence our claim follows.

From Theorem \ref{thm:sufficient}, for all such surfaces $\sigma_1$, for every $x_0\in \Omega$ there exists a $C^1$ phase discontinuity $\phi$ in a neighborhood of $(Q(x_0),a)\in \sigma_2$ so that the collimated constant field ${\bf e}$ emitted from $(x,0)$ with $x$ in a neighborhood of $x_0$, is refracted by the lens $(\sigma_1,(\sigma_2,\phi))$ into the vertical direction ${\bf w}=(0,0,1)$. 

%In the next section, we will investigate more closely the case of a collimated vertical incident field ${\bf e}$, i.e. ${\bf e}={\bf w}=(0,0,1).$
\end{remark}

%{\color{red} WE STOPPED HERE}
\subsection{The case of a vertical incident field.} \label{subsec:vertical}
We assume in this section that the incident field emitted from $\Omega$ is vertical, i.e. ${\bf e}(x)={\bf w}=(0,0,1)$ for every $x\in \Omega$. In this case, $P(x)=(x,\rho(x))$, and so $\sigma_1$ is the graph of the function $\rho$. The normal vector at each point $P(x)$ is then given by $\nu(x)=\dfrac{(-D\rho(x),1)}{\sqrt{1+|D\rho(x)|^2}}.$
%Notice that ${\bf e}\cdot {\nu}=\dfrac{1}{1+|D\rho|^2}>0$, then we get from Section \ref{eq: standard Snell's law} that
%\begin{equation}\label{eq:strict refraction}
%m_3(x)={\bf e}\cdot {\bf m}(x)>\dfrac{1}{\kappa_1}.
%\end{equation}

We simplify the formulas for ${\bf m}$ and $d$ given respectively in \eqref{eq: m inside II} and \eqref{eq: formula for d general}. Denoting
\begin{equation}\label{eq: Delta}
\Delta(x)=\sqrt{\kappa_1^2+(\kappa_1^2-1)|D\rho(x)|^2},
\end{equation}
\eqref{eq: lambda} yields
$$\lambda=\dfrac{1-\kappa_1^2}{(0,0,1)\cdot {\bf \nu}+\sqrt{\kappa_1^2-1+((0,0,1)\cdot \nu)^2}}=\dfrac{(1-\kappa_1^2)\sqrt{1+|D\rho|^2}}{1+\Delta},$$
and so replacing in \eqref{eq: m inside II}

\begin{equation}\label{eq:m for vertical field}
{\bf m}=\dfrac{1}{\kappa_1}\left((0,0,1)-\dfrac{1-\kappa_1^2}{1+\Delta}(-D\rho,1)\right)=\dfrac{1}{\kappa_1}\left(\dfrac{1-\kappa_1^2}{1+\Delta}D\rho, 1+\dfrac{\kappa_1^2-1}{1+\Delta}\right).
\end{equation}
Therefore from \eqref{eq: formula for d general}
\begin{equation}\label{eq: d for vertical field}
d(x)=\dfrac{\kappa_1(a-\rho(x))(1+\Delta(x))}{\kappa_1^2+\Delta(x)}.
\end{equation}

Since, in Theorem \ref{thm:sufficient}, we need ${\bf m}$ to be $C^2$, we assume that $\rho$ is $C^3$ and prove the following theorem.
\begin{theorem}\label{thm:sufficient for vertical field}
Given $\rho\in C^3(\Omega)$, and ${\bf e}(x)=(0,0,1),$ then condition \eqref{eq:big det} is satisfied at $x_0$ if and only if $\det D^2\rho\neq 0$ and
\begin{equation}\label{eq: det vertical}
\det \left(I+\dfrac{\kappa_1^2-1}{\kappa_1^2}D\rho\otimes D\rho+\dfrac{(a-\rho)(1-\kappa_1^2)}{\kappa_1^2+\Delta}D^2\rho\right)\neq 0
\end{equation}
at $x_0$.
%Define $\Lambda_1,\Lambda_2$ the eigenvalues of $D^2\rho$, with $\Lambda_1\leq \Lambda_2$, then if either of the following holds:
%\begin{enumerate}
%\item $\lambda_1\leq \lambda_2<0$.
%\item $\cdots$
%\item $\cdots$
%\end{enumerate}
%\eqref{det for vertical} is satisfied, and then there exists a phase discontinuity $\phi$ so that the rays ${\bf m}(x)$ in medium II given by \eqref{eq:m for vertical field} are refracted into ${\bf w}=(0,0,1)$ by the metasurface $(\sigma_2,\phi)$.
\end{theorem}

\begin{proof} The objective is to simplify \eqref{eq:det for collimated}. First, from \eqref{eq:m for vertical field}
\begin{equation}\label{eq:1-kappam3}
1-\kappa_1  {\bf e}\cdot {\bf m}=1-\kappa_1m_3=\dfrac{1-\kappa_1^2}{1+\Delta}.
\end{equation}
Second, differentiating \eqref{eq:m for vertical field} with respect to $x_i$, $i=1,2$ yields
$${\bf m}_{x_i}=\dfrac{1-\kappa_1^2}{\kappa_1(1+\Delta)}\left((D\rho_{x_i},0)+\dfrac{\Delta_{x_i}}{1+\Delta}(-D\rho,1)\right).$$
From \eqref{eq: Delta}, $\Delta_{x_i}=\dfrac{(\kappa_1^2-1)}{\Delta}D\rho_{x_i}\cdot D\rho,$
and so
\[
{\bf m}_{x_i}=\dfrac{1-\kappa_1^2}{\kappa_1(1+\Delta)}\left[(D\rho_{x_i},0)+\dfrac{(\kappa_1^2-1)D\rho_{x_i}\cdot D\rho}{\Delta(1+\Delta)}(-D\rho,1)\right]
.\]
For $i,j=1,2$
\begin{align*}
    {\bf m}_{x_i}\cdot {\bf m}_{x_j}&=\dfrac{(\kappa_1^2-1)^2}{\kappa_1^2(1+\Delta)^2}\Big[D\rho_{x_i}\cdot D\rho_{x_j}-2\dfrac{(\kappa_1^2-1)}{\Delta(1+\Delta)}(D\rho_{x_i}\cdot D\rho)(D\rho_{x_j}\cdot D\rho)\\ &\qquad \qquad\qquad\qquad\qquad+\dfrac{(\kappa_1^2-1)^2(D\rho_{x_i}\cdot D\rho)(D\rho_{x_j}\cdot D\rho)}{\Delta^2(1+\Delta)^2}(1+|D\rho|^2)\Big]\\
    &=\dfrac{(\kappa_1^2-1)^2}{\kappa_1^2(1+\Delta)^2}\left[D\rho_{x_i}\cdot D\rho_{x_j}+\dfrac{\kappa_1^2-1}{\Delta(1+\Delta)}(D\rho_{x_i}\cdot D\rho)(D\rho_{x_j}\cdot D\rho)\left(\dfrac{(\kappa_1^2-1)(1+|D\rho|^2)}{\Delta(1+\Delta)}-2\right)\right]
\end{align*}
From \eqref{eq: Delta}, $(\kappa_1^2-1)|D\rho|^2=\Delta^2-\kappa_1^2$, and hence
$${\bf m}_{x_i}\cdot {\bf m}_{x_j}=\dfrac{(\kappa_1^2-1)^2}{\kappa_1^2(1+\Delta)^2}\left[D\rho_{x_i}\cdot D\rho_{x_j}-\dfrac{\kappa_1^2-1}{\Delta^2}(D\rho_{x_i}\cdot D\rho)(D\rho_{x_j}\cdot D\rho)\right].$$
Therefore
\begin{equation}\label{Dm times Dm}
D{\bf m}\otimes D{\bf m}=\dfrac{(\kappa_1^2-1)^2}{\kappa_1^2(1+\Delta)^2}\left[(D^2\rho)^2-\dfrac{\kappa_1^2-1}{\Delta^2}D^2\rho(D\rho\otimes D\rho)D^2\rho\right]=\dfrac{(\kappa_1^2-1)^2}{\kappa_1^2(1+\Delta)^2}D^2\rho\left[I-\dfrac{\kappa_1^2-1}{\Delta^2}D\rho\otimes D\rho\right]D^2\rho
\end{equation}

Replacing \eqref{eq: d for vertical field}, \eqref{eq:1-kappam3}, and \eqref{Dm times Dm}  in \eqref{eq:det for collimated}, the sufficient condition becomes
$$ 0\neq \left(\det D^2\rho\right)\det\left(\dfrac{1-\kappa_1^2}{1+\Delta}I+ \dfrac{(\kappa_1^2-1)^2(a-\rho)}{(1+\Delta)(\kappa_1^2+\Delta)}\left[I-\dfrac{\kappa_1^2-1}{\Delta^2}D\rho\otimes D\rho\right]D^2\rho\right).$$
Let
$$\mathcal M=\dfrac{1-\kappa_1^2}{1+\Delta}I+ \dfrac{(\kappa_1^2-1)^2(a-\rho)}{(1+\Delta)(\kappa_1^2+\Delta)}\left[I-\dfrac{\kappa_1^2-1}{\Delta^2}D\rho\otimes D\rho\right]D^2\rho,$$
then \eqref{eq:det for collimated} is satisfied for ${\bf e}=(0,0,1)$ if and only if at $x=x_0$, $D^2\rho$ and $\mathcal M$ are invertible. 

We simplify $\det \mathcal M$.
Let $\mathcal W=I-\dfrac{\kappa_1^2-1}{\Delta^2}D\rho\otimes D\rho$, the matrix determinant Lemma \cite[Lemma 1.1]{DING20071223} and \eqref{eq: Delta} imply that
\begin{align*}
    \det \mathcal W=1-\dfrac{(\kappa_1^2-1)|D\rho|^2}{\Delta^2}=\dfrac{\kappa_1^2}{\Delta^2},
\end{align*}
then $\mathcal W$ is invertible and by the Sherman-Morrison formula \cite[Chapter III.1]{strang2019linear}
$$\mathcal W^{-1}=I+\dfrac{\kappa_1^2-1}{\kappa_1^2}D\rho\otimes D\rho.$$
Therefore,
\begin{align*}
\mathcal M=\dfrac{1-\kappa_1^2}{1+\Delta}\left(I+\dfrac{(1-\kappa_1^2)(a-\rho)}{\kappa_1^2+\Delta}\mathcal W D^2\rho\right)=\dfrac{1-\kappa_1^2}{1+\Delta}\mathcal W\left(I+\dfrac{\kappa_1^2-1}{\kappa_1^2}D\rho\otimes D\rho+\dfrac{(1-\kappa_1^2)(a-\rho)}{\kappa_1^2+\Delta}D^2\rho\right).
\end{align*}
We conclude that $\mathcal M$ is invertible at $x_0$ if and only if \eqref{eq: det vertical} is satisfied, completing hence the proof of the theorem.
\end{proof}

From Remark \ref{rmk:collimated}, $D^2\rho(x_0)<0$ is sufficient for existence of a lens $(\sigma_1,(\sigma_2,\phi))$ in a neighborhood of $x_0$ refracting a collimated field into the vertical direction. However, from Theorem \ref{thm:sufficient for vertical field}, this condition can be relaxed in the case when ${\bf e}=(0,0,1)$ allowing a larger family of lower faces $\sigma_1$, as summarized in the following corollary.

%From Remark \ref{rmk:collimated}}, in the case when ${\bf e}=(0,0,1)$, we have that if $D^2\rho<0$ and so $\rho$ is strictly concave function then the sufficient condition in Theorem \ref{sufficient} for the existence of phase discontinuity is satisfied, this follows easily from Theorem \ref{}, in fact, more cases can be also deduced.

\begin{corollary}
Assume $D^2\rho(x_0)$ is invertible, and $\Lambda_1,\Lambda_2$ its corresponding non-zero eigenvalues with $\Lambda_1\geq \Lambda_2$. If 
\begin{equation}\label{ineq: eigenvalues}
\Lambda_2> \dfrac{\Delta^2(x_0)(\kappa_1^2+\Delta(x_0))}{\kappa_1^2(\kappa_1^2-1)(a-\rho(x_0))},\qquad or\qquad \Lambda_1<\dfrac{\kappa_1^2+\Delta(x_0)}{(\kappa_1^2-1)(a-\rho(x_0))}.
\end{equation}
then
\eqref{eq: det vertical} is satisfied at $x_0$.
\end{corollary}
\begin{proof}
We denote by $\mathcal A$ the matrix in \eqref{eq: det vertical}, i.e.
\begin{equation}\label{eq:mathcal A}
\mathcal A=I+\dfrac{\kappa_1^2-1}{\kappa_1^2}D\rho\otimes D\rho+\dfrac{(1-\kappa_1^2)(a-\rho)}{\kappa_1^2+\Delta}D^2\rho,
\end{equation}
and let $\mu_1,\mu_2$ be its corresponding eigenvalues with $\mu_1\geq \mu_2$.
Condition \ref{eq: det vertical} is satisfied if and only if $\mathcal A$ is invertible i.e. $\mu_1,\mu_2\neq 0$.

The eigenvalues of the matrix $I+\dfrac{\kappa_1^2-1}{\kappa_1^2}D\rho\otimes D\rho$ are  in decreasing order $\dfrac{\Delta^2}{\kappa_1^2}$ and $1$. Since $\kappa_1>1$, $a-\rho>0$, and $\Lambda_1\geq \Lambda_2$ then the eignevalues of $\dfrac{(1-\kappa_1^2)(a-\rho)}{\kappa_1(1+\Delta)}D^2\rho$ are in decreasing order
$\dfrac{(1-\kappa_1^2)(a-\rho)}{\kappa_1^2+\Delta}\Lambda_2$ and $\dfrac{(1-\kappa_1^2)(a-\rho)}{\kappa_1^2+\Delta}\Lambda_1.$

Hence, from Weyl's inequality 
$$1+\dfrac{(1-\kappa_1^2)(a-\rho)}{\kappa_1^2+\Delta}\Lambda_1\leq \mu_2\leq \mu_1\leq \dfrac{\Delta^2}{\kappa_1^2}+\dfrac{(1-\kappa_1^2)(a-\rho)}{\kappa_1^2+\Delta}\Lambda_2.$$
Therefore, inequalities \ref{ineq: eigenvalues} imply that $\mathcal A$ is invertible.
%Therefore
\end{proof}

\begin{remark}\rm
The first inequality in \eqref{ineq: eigenvalues} implies that $\Lambda_1\geq \Lambda_2>0$ and so $\rho$ is strictly convex in a neighborhood of $x_0$. The second inequality in \eqref{ineq: eigenvalues} allows cases where $D^2\rho(x_0)$ is positive definite or negative definite or indefinite.
\end{remark}

\section{A near-field imaging problem}\label{sec:Imaging}

%In this section, we introduce an imaging problem and formulate the corresponding system of partial differential equations. We find necessary and sufficient conditions for the imaging map $T$ to solve this system. Furthermore, we bind the imaging problem and the acquired conditions in this regard with the problem of the existence of metasurface, investigating any crossing of conditions and examining some specific transformations.
\subsection{Problem setup}\label{subsec: setup imaging} 
We are given $\Omega,\Omega^*\subseteq \mathbb R^2$ open and simply connected domains, a $C^1$ bijective map $T:\Omega\mapsto \Omega^*$, positive real numbers $a$ and $c$ with $c>a$.
   \begin{figure}[h]
\centering
\includegraphics[width=0.6\textwidth]{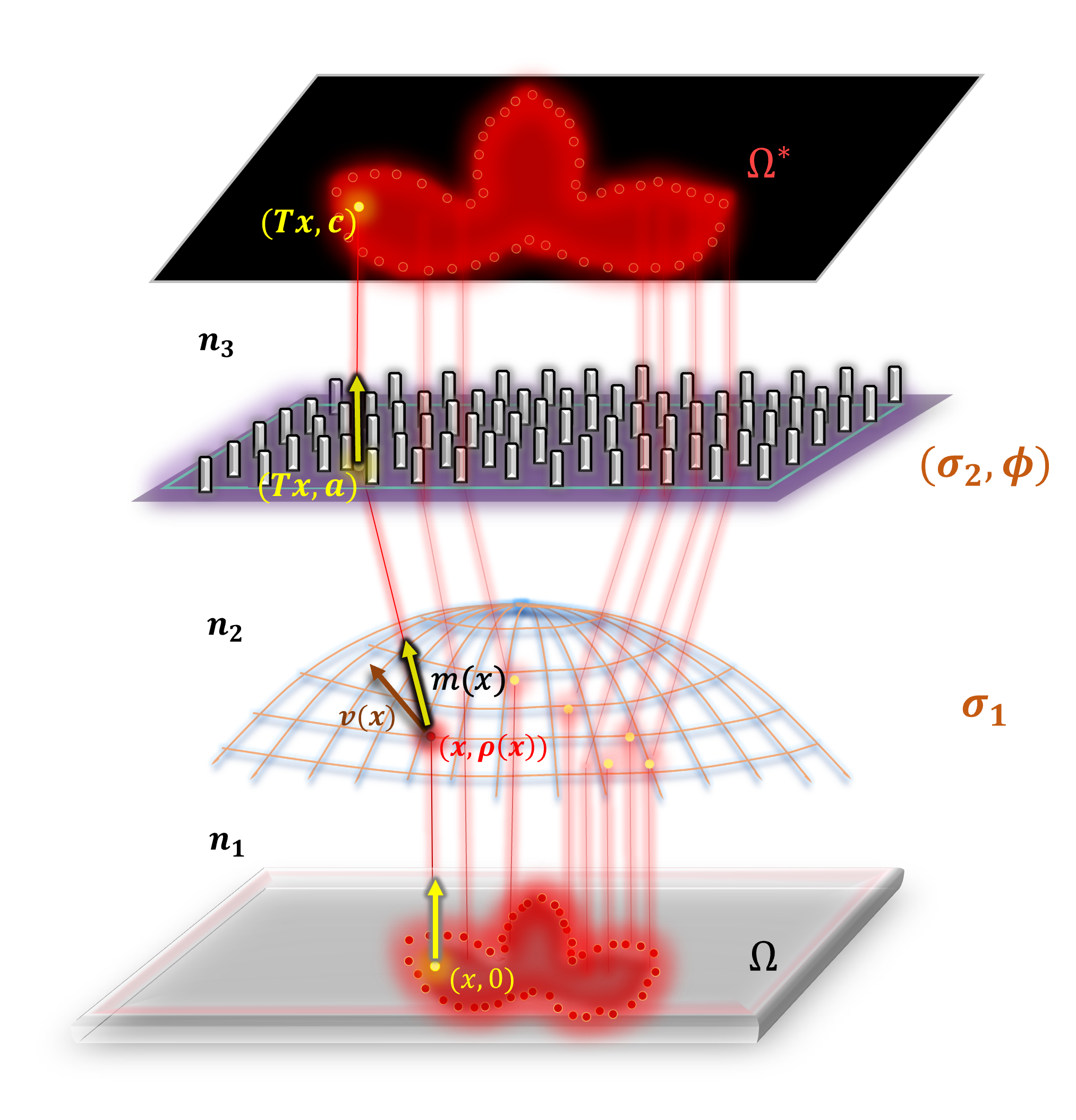}
\caption{}
\label{Imaging Problem}
\end{figure}
 Monochromatic radiation with frequency $\omega$ are issued from $(x,0), x=(x_1,x_2)\in \Omega$, with vertical direction ${\bf e}=(0,0,1)$. Our goal is to construct a lens $(\sigma_1,(\sigma_2,\phi))$ with lower face $\sigma_1=\{(x,\rho(x))\}_{x\in \Omega}$ a conventional $C^2$ refracting surface between the planes $\{x_3=0\}$ and $\{x_3=a\}$; and upper face $(\sigma_2,\phi)$ a planar metasurface with $\sigma_2\subseteq \{x_3=a\}$ and  $\phi:=\phi(u_1,u_2,u_3)$ a tangential phase discontinuity defined in a neighborhood of every point of $\sigma_2$, such that:

\begin{enumerate}
    \item the lens $(\sigma_1,(\sigma_2,\phi))$ refracts every incident vertical ray emitted from $(x,0)$, $x\in \Omega$, into the point $(Tx,c)$;
    \item all the rays leave $(\sigma_2,\phi)$ with the vertical direction.
\end{enumerate} 
 The lens $(\sigma_1,(\sigma_2,\phi))$ projects an image $\Omega^*=T(\Omega)$ on the target plane $\{x_3=c\}$,  see Figure \ref{Imaging Problem}.
 
 We assume that $n_1, n_2$ and $n_3$ with $n_2>n_1$ are the refractive indices of the media I, II and III corresponding to the regions below, enclosed by, and above the lens $(\sigma_1,(\sigma_2,\phi))$ respectively.

%As shown in Section \ref{sec:general field}, 
%the existence of a metasurface lens $(\sigma_2,\phi)$ depends on $\sigma_1$ satisfying \eqref{}. 

To do this, we first investigate the existence of $\rho$ positive and $C^2$ such that for every $x\in \Omega$ the surface $\sigma_1=\{(x,\rho(x))\}_{x\in \Omega}$ refracts the incident vertical ray emitted from $(x,0)$
into the point $(Tx,a)$ on $\sigma_2$, and then verify whether $\rho$ satisfies the conditions in Theorem \ref{thm:sufficient for vertical field} when it exists.

Throughout this section, denote $\kappa_1=\dfrac{n_2}{n_1}$ and $Sx=Tx-x$.

\subsection{Existence of $\rho$}\label{sec: existence of rho}

\begin{proposition}\label{prop: PDE of rho}
A $C^2$ surface $\sigma_1=\{(x,\rho(x))\}_{x\in \Omega}$, separating media I and II, and refracting vertical rays emitted from $(x,0)$ in medium I into $(Tx,a)$ in medium II exists if and only if %for every $x\in \Omega$, $\rho(x)>0$, 
\begin{equation}\label{eq:imaging inequality}
a>a-\rho(x)>\dfrac{|Sx|}{\sqrt{\kappa_1^2-1}}
\end{equation}
for every $x\in \Omega$,
and $\rho$ satisfies the following system of PDEs
\begin{equation}\label{eq:PDE of rho}
D\rho(x)=\dfrac{\kappa_1 Sx}{\sqrt{|Sx|^2+(a-\rho(x))^2}-\kappa_1(a-\rho(x))}.
\end{equation}
\end{proposition}

\begin{proof}
Assume for each $x\in \Omega$ the ray with vertical direction ${\bf e}=(0,0,1)$ is refracted at $(x,\rho(x))$ into the point $(Tx,a)$, then trivially $\sigma_1$ is between the plane $\{x_3=0\}$ and $\{x_3=a\}$ and so $0<a-\rho<a$. The unit direction of the refracted ray is then
\begin{equation*}%\label{eq: m imaging}
{\bf m}(x)=\dfrac{(Tx,a)-(x,\rho(x))}{|(Tx,a)-(x,\rho(x))|}=\dfrac{(Sx,a-\rho(x))}{|(S(x),a-\rho(x))|}.
\end{equation*}
Therefore, from \eqref{eq:m for vertical field} 
\begin{equation}\label{eq:components of m}
\begin{cases}
    \dfrac{Sx}{|(Sx,a-\rho(x))|}&=\dfrac{1-\kappa_1^2}{\kappa_1(1+\Delta)}D\rho(x) \\
    \dfrac{a-\rho(x)}{|(Sx,a-\rho(x))|}&=\dfrac{1}{\kappa_1}\left(1+\dfrac{\kappa_1^2-1}{1+\Delta}\right) 
\end{cases}
\end{equation}
with $\Delta$ given in \eqref{eq: Delta}.

From the second equation of system \eqref{eq:components of m}, $\dfrac{a-\rho(x)}{|(Sx,a-\rho(x))|}>\dfrac{1}{\kappa_1}$. Squaring both sides and solving for $a-\rho$ yields the right inequality in \eqref{eq:imaging inequality}.
Further, this same equation gives
\begin{equation}\label{eq:auxiliary}
\dfrac{1-\kappa_1^2}{1+\Delta}=\dfrac{|(Sx,a-\rho(x))|-\kappa_1(a-\rho(x))}{|(Sx,a-\rho(x))|},
\end{equation}
and so replacing \eqref{eq:auxiliary} in the first equation of the system \eqref{eq:components of m} and solving for $D\rho$ we obtain \eqref{eq:PDE of rho}.

Conversely, assume $\rho$ satisfies \eqref{eq:imaging inequality}, and \eqref{eq:PDE of rho}. From \eqref{eq:imaging inequality}, $\sigma_1=\{(x,\rho(x)\}$ is between the plane $\{x_3=0\}$ and $\{x_3=a\}$,  and since $S\in C^1(\Omega)$ then from \eqref{eq:PDE of rho} $\rho\in C^2(\Omega).$  We show that $\rho$ verifies the system \eqref{eq:components of m}, which implies from \eqref{eq:m for vertical field} that the vertical ray emitted from $(x,0)$ is refracted at $(x,\rho(x))$ into $(Tx,a)$.

In fact, from \eqref{eq:imaging inequality}, the denominator in \eqref{eq:PDE of rho} is negative, so replacing \eqref{eq:PDE of rho} in \eqref{eq: Delta} yields
%{\footnotesize
\begin{align*}
\Delta&=\sqrt{\kappa_1^2+\dfrac{(\kappa_1^2-1)\kappa_1^2|Sx|^2}{\left(|(Sx,a-\rho(x))|-\kappa_1(a-\rho(x))\right)^2}}\\
&=\dfrac{\kappa_1\sqrt{|(Sx,a-\rho(x))|^2+\kappa_1^2(a-\rho(x))^2-2\kappa_1|(Sx,a-\rho(x))|(a-\rho(x))+(\kappa_1^2-1)|Sx|^2}}{\kappa_1(a-\rho(x))-|(Sx,a-\rho(x))|}\\
&=\dfrac{\kappa_1\sqrt{\kappa_1^2|(Sx,a-\rho(x))|^2-2\kappa_1|(Sx,a-\rho(x))|(a-\rho(x))+(a-\rho(x))^2}}{\kappa_1(a-\rho(x))-|(Sx,a-\rho(x))|}\\
\end{align*}
Since $\kappa_1>1$, then $\kappa_1|(Sx,a-\rho(x))|\geq a-\rho(x)$, and so
\begin{equation}\label{eq:nice Delta}
\Delta(x)=\dfrac{\kappa_1^2|(Sx,a-\rho(x))|-\kappa_1(a-\rho(x))}{\kappa_1(a-\rho(x))-|(Sx,a-\rho(x))|}=-1+\dfrac{(\kappa_1^2-1)|(Sx,a-\rho(x))|}{\kappa_1(a-\rho(x))-|(Sx,a-\rho(x))|}.
\end{equation}
We then obtain \eqref{eq:auxiliary} and so together with \eqref{eq:PDE of rho}, the system \eqref{eq:components of m} follows.

%}

%The unit normal vector to $\sigma_1$ at $(x,\rho(x))$ toward medium II is given by $\nu(x)=\dfrac{(-D\rho(x),1)}{\sqrt{1+|D\rho(x)|}}$, moreover from \eqref{eq:PDE of rho}
%\begin{align*}
 %   {\bf e}-\kappa_1\dfrac{(Sx,a-\rho(x))}{|(Sx,a-\rho(x))|}&=\dfrac{1}{|(Sx,a-\rho(x))|}\left(-\kappa_1 Sx, |(Sx, a-\rho(x))|-\kappa_1(a-\rho(x))|\right)\\
 %   &=\dfrac{|(Sx,a-\rho(x))|-\kappa_1(a-\rho(x))}{|(Sx,a-\rho(x))|}(-D\rho(x),1),
%\end{align*}
%and hence 
%$${\bf e}\times \nu=\kappa_1 \dfrac{(Sx,a-\rho(x))}{|(Sx,a-\rho(x))|}\times \nu.$$
%Moreover, inequality \eqref{eq:imaging inequality} implies that
%$${\bf e}\cdot \dfrac{(Sx,a-\rho(x))}{|(Sx,a-\rho(x))|}=\dfrac{a-\rho(x)}{|(Sx,a-\rho(x))|}>\dfrac{1}{\kappa_1}.$$
%Therefore, Snell's law \eqref{sec:Standard Snell}, and the compatibility condition \eqref{cond: kappa>1} implies that the vertical ray is refracted at $(x,\rho(x))$ into the unit direction $\dfrac{(Sx,a-\rho(x))}{a-\rho(x)}$ and hence striking the plane $\{u_3=a\}$ at the point $(Sx,a)$.
\end{proof}

From Inequality \eqref{eq:inequality local}, given the map $T$, $Sx=Tx-x$, the plane $\{x_3=a\}$ should be chosen so that 
\begin{equation}\label{eq:thickness}
a>\dfrac{|Sx|}{\sqrt{\kappa_1^2-1}}
\end{equation}
for every $x$, this gives a condition on the thickness of our objective lens. In that case, finding the lower face $\sigma_1=\{(x,\rho(x))\}_{x\in \Omega}$ of the lens solution to the imaging problem in Section \ref{subsec: setup imaging} reduces to finding positive solutions to the system \eqref{eq:PDE of rho} satisfying the inequality \eqref{eq:imaging inequality}. 
% We
%find conditions on $S$ For that, assume that $\Omega$ is also simply connected.
Notice that \eqref{eq:PDE of rho} can be written as follows
\begin{equation}\label{eq:form 2 of PDE}
D(a-\rho(x))=\dfrac{\kappa_1\dfrac{Sx}{a-\rho(x)}}{\kappa_1-\sqrt{\left|\dfrac{Sx}{a-\rho(x)}\right|^2+1}}:=V\left(x,a-\rho(x)\right)\end{equation}
with $V(x,z):=(V_1(x,z),V_2(x,z))=\dfrac{\kappa_1 \dfrac{Sx}{z}}{\kappa_1-\sqrt{\left|\dfrac{Sx}{z}\right|^2+1}},$
$x=(x_1,x_2)\in \Omega$, and $z\in \mathbb R$ such that $a>z>\dfrac{|Sx|}{\sqrt{\kappa_1^2-1}}.$
Writing \eqref{eq:PDE of rho} in the form \eqref{eq:form 2 of PDE} allows us to use the theory in \cite[Chapter 6]{hartman2002ordinary} to find necessary and sufficient conditions for the existence and uniqueness of local solutions $\rho$. 
In fact, since $S$ is $C^1$ then a solution $\rho$ to \eqref{eq:form 2 of PDE} is $C^2$ and by the mixed derivative theorem $\rho_{x_1x_2}=\rho_{x_2x_1}$. Hence,% from \eqref{eq:form 2 of PDE} and the chain rule that
\begin{equation}\label{eq:Hartman condition}
\dfrac{\partial V_1}{\partial x_2} +\dfrac{\partial V_1}{\partial z}V_2=\dfrac{\partial V_2}{\partial x_1}+\dfrac{\partial V_2}{\partial z}V_1.
\end{equation}
%Conversely, from CITE!!! \eqref{eq:Hartman condition} is sufficient for the existence of a unique local solution.

%Before stating the main result , recall that for two-dimensional fields ${\bf a}(x)=(a_1(x),a_2(x))$ and ${\bf b}(x)=(b_1(x),b_2(x))$, the two-dimensional cross product is given by
%$${\bf a}(x)\times {\bf b}(x)=\left|\begin{matrix} a_1(x) & a_2(x)\\ b_1(x) & b_2(x)\end{matrix}\right|.$$
\begin{theorem}\label{thm:Hartman verification}
Given $x_0\in \Omega$ satisfying \eqref{eq:thickness} and $z_0$ such that
\begin{equation}\label{eq:inequality local}
   a>z_0>\dfrac{|Sx_0|}{\sqrt{\kappa_1^2-1}},
\end{equation}
\eqref{eq:Hartman condition} 
is satisfied in an open neighborhood $\mathcal \mathcal U\subseteq \left\{(x,z):x\in \Omega, a>z>\frac{|Sx|}{\sqrt{\kappa_1^2-1}}\right\}$ of $(x_0,z_0)$ if and only if for every $x$ in a neighborhood of $x_0$
\begin{align}
\nabla \times S&=0\label{condition one}
\\
S\times D|S|^2&=0.\label{condition two}
\end{align}
%there exists 
%Assume $T:\Omega\mapsto \Omega^*$ is a $C^1$ bijective map, $\Omega$ a simply connected domain, and $Sx=Tx-x$.
%Given $(x_0,z_0)$ with $x_0\in \Omega$ and 
 %of $x_0$ and $V$ of $z_0$ such that $U\times V\subseteq \{(x,z)\in \Omega: x\in \Omega, z>\frac{|Sx|}{\sqrt{\kappa_1^2-1}}\}$ if and only if for every $x\in U$
%\eqref{eq:Hartman condition} is satisfied in some open set $U\times V\subset \left\{(x,z): x\in \Omega, |z|>\dfrac{Sx}{\sqrt{\kappa_1^2-1}}\right\}$  if and only if at every $x\in U$
%\begin{align}\label{condition one} 
%    \nabla \times S&=0\\ 
 %   S\times D(|S|^2)&=0 \label{condition two}
%\end{align}
%As a consequence, if $x_0\in \Omega$, such that \eqref{condition one} and \eqref{condition two} are satisfied in an open set containing $x_0$, and $z_0$ is a real number satisfying \eqref{eq:inequality local}
%\begin{equation}\label{eq:inequality local}
 %   z_0>\dfrac{|Sx_0|}{\sqrt{\kappa_1^2-1}}
%\end{equation}
%and such that \eqref{condition one} and \eqref{condition two} are satisfied in an open set containing $x_0$ 
%a unique solution $\rho$ to \eqref{eq:PDE of rho} that is $C^2$ in a neighborhood $\mathcal O_{x_0}$ of $x_0$ such that $a-\rho(x_0)=z_0$ and satisfying \eqref{eq:imaging inequality} for every $x\in \mathcal O_{x_0}$.
\end{theorem}

\begin{proof} 
Write $S=(S_1,S_2)$, and $V_i(x,z)=\varphi\left(\left|\dfrac{Sx}{z}\right|^2\right)\dfrac{S_ix}{z},$
with 
\begin{equation}\label{eq:psi}
\varphi(y)=\dfrac{\kappa_1}{\kappa_1-\sqrt{y+1}}
\end{equation}
defined for $y\in \left[0,\kappa_1^2-1\right)$. Notice that $\varphi$ is positive, increasing and 
\begin{equation}\label{eq:phi'}
\varphi'(y)=\dfrac{\kappa_1}{2\sqrt{y+1}\left(\kappa_1-\sqrt{y+1}\right)^2}=\dfrac{1}{2\kappa_1\sqrt{y+1}}\varphi^2(y).
\end{equation}
Therefore for $i,j=1,2$
\begin{align}
\dfrac{\partial V_i}{\partial x_j}&=2\dfrac{S_ix}{z^3}(S\cdot S_{x_j})\varphi'\left(\left|\dfrac{Sx}{z}\right|^2\right)+\dfrac{1}{z}\varphi\left(\left|\dfrac{Sx}{z}\right|^2\right)\dfrac{\partial S_i}{\partial x_j} \label{eq:vixj}\\
\dfrac{\partial V_i}{\partial z}&=-\dfrac{2|Sx|^2}{z^4}\varphi'\left(\left|\dfrac{Sx}{z}\right|^2\right)S_ix-\dfrac{1}{z^2}\varphi\left(\left|\dfrac{Sx}{z}\right|^2\right)S_ix .\label{eq:Viz}
\end{align}
Notice that $\dfrac{\partial V_1}{\partial z}V_2=\dfrac{\partial V_2}{\partial z}V_1$, and hence
\eqref{eq:Hartman condition} becomes
$$2\dfrac{S_1x}{z^3}(S\cdot S_{x_2})\varphi'\left(\left|\dfrac{Sx}{z}\right|^2\right)+\dfrac{1}{z}\varphi\left(\left|\dfrac{Sx}{z}\right|^2\right)(S_1)_{x_2}=
%\dfrac{\partial S_1}{\partial x_2}=
2\dfrac{S_2x}{z^3}(S\cdot S_{x_1})\varphi'\left(\left|\dfrac{Sx}{z}\right|^2\right)+\dfrac{1}{z}\varphi\left(\left|\dfrac{Sx}{z}\right|^2\right)(S_2)_{x_1},$$
and so
\begin{equation}\label{eq: Hartman condition bis}
  %  0&=\dfrac{2}{z^2}((S\cdot S_{x_2})S_1-(S\cdot S_{x_1})S_2)\varphi'+((S_1)_{x_2}-(S_{2})_{x_1})\varphi \notag\\
    \dfrac{1}{z^2}\left(S\times D|S|^2\right)\varphi'\left(\left|\dfrac{Sx}{z}\right|^2\right)-(\nabla\times S)\varphi\left(\left|\dfrac{Sx}{z}\right|^2\right)=0.
   % &=\dfrac{1}{z^2}\dfrac{1}{2\kappa_1\sqrt{\dfrac{|S|^2}{z^2}}+1}\left(S\times D|S|^2\right)\varphi^2-(\nabla\times S)\varphi
\end{equation}
Clearly \eqref{condition one}, and \eqref{condition two} implies \eqref{eq: Hartman condition bis}. 

Conversely, assume \eqref{eq: Hartman condition bis} is satisfied for every $x$ in a neighborhood $U_{x_0}$ of $x_0$ in $\Omega$ and $z$ in a neighborhood $V_{z_0}$ of $z_0$ such that $a>z>\dfrac{|Sx|}{\sqrt{\kappa_1^2-1}}$ for every $x\in U_{x_0}$ and $z\in V_{z_0}$. We show that \eqref{condition one} and \eqref{condition two} then follow for every $x\in U_{x_0}$. In fact, from \eqref{eq:phi'}, and the fact that $\varphi>0$, \eqref{eq: Hartman condition bis} can be written as follows
$$\dfrac{1}{2\kappa_1 z^2\sqrt{1+\left|\dfrac{Sx}{z}\right|^2}}(S\times D|S|^2)\varphi\left(\left|\dfrac{Sx}{z}\right|^2\right)-\nabla \times S=0.$$
Fixing $x\in U_{x_0}$, and differentiating with respect to $z$ yields
$$\dfrac{\partial}{\partial z}\left(\dfrac{1}{2\kappa_1 z^2\sqrt{1+\left|\dfrac{Sx}{z}\right|^2}}\varphi\left(\left|\dfrac{Sx}{z}\right|^2\right)\right)(S\times D|S|^2)=0\qquad\forall z\in V_{z_0}.$$
Since the term in large parenthesis above is not constant in $z$ then $(S\times D|S|^2) (x)=0$, and therefore from \eqref{eq: Hartman condition bis}, since $\varphi>0$, $\nabla \times S (x)=0$.

%suppose for some $\hat x$, $(S\times D|S|^2)(\hat x)\neq 0$ then particularly $S\hat x\neq 0$. From \eqref{eq: Hartman condition bis} the fact that $\varphi\neq 0$ we get that for every $z\in V_{z_0}$
%$$\dfrac{(\nabla \times S) ( \hat x)}{(S\times |DS|^2) (\hat x)}=\dfrac{\varphi'\left(\dfrac{S\hat x}{z}\right)}{\varphi\left(\dfrac{S\hat x}{z}\right)},$$
%a contradiction since the righthand is not constant in any open set containing $z_0$. We hence obtain\eqref{condition two}, and \eqref{condition one} hence follows from \eqref{eq: Hartman condition bis}  %\eqref{condition one}

%since $\varphi$ and $\varphi'$ are positive suppose \eqref{eq: Hartman condition bis} is satisfied in some open neighborhood of a point $(x_0,z_0)$$U\times V$, then \eqref{condition one} and \eqref{condition two} are equivalent since $\varphi$ and $\varphi'$ are positive
%Let $x_0\in U$. Notice that if $Sx_0=0$ then \eqref{condition two} is immediate and replacing in \eqref{eq: Hartman condition bis} we get \eqref{condition one}. Assume there exists $x_0\in U$ such that $(S\times DS)(x_0)\neq 0$, then $Sx_0\neq 0$ and for every $z\in V$
%$$\dfrac{\nabla\times S}{S\times D|S|^2}(x_0)=\dfrac{1}{z^2}\dfrac{\varphi'\left(\left|\dfrac{Sx_0}{z}\right|^2\right)}{\varphi\left(\left|\dfrac{Sx_0}{z}\right|^2\right)},$$
%a contradiction. %since the right-hand side of the equality is not constant in $z$, and similarly we obtain a contradiction if $(S\times D|S|^2)(x_0)\neq 0$.
\end{proof}

Consequently Theorem \ref{thm:Hartman verification} and \cite[Chapter 6]{hartman2002ordinary} implies the following result.

\begin{corollary}\label{cor:result}
Given $x_0\in \Omega$ satisfying \eqref{eq:thickness} and $z_0$ verifying \eqref{eq:inequality local}, if \eqref{condition one} and \eqref{condition two} hold for every $x$ in an open set containing $x_0$, then the system of PDEs \eqref{eq:PDE of rho} has a unique positive $C^2$ solution $\rho$ in a neighborhood of $x_0$ satisfying \eqref{eq:imaging inequality} with $a-\rho(x_0)=z_0$.
\end{corollary}

\begin{remark}\label{rmk:class maps}\rm Condition \eqref{condition one} is equivalent to say that $DS$ is a symmetric matrix. Condition \eqref{condition two} is equivalent to say that $(DS)(S^t)$ is parallel to $S^t$. Hence if $Sx\neq (0,0)$, $S^tx$ is an eigenvector of $DS(x)$. In this case, by the spectral theorem $S^t_{\perp}x=\begin{pmatrix} -S_2x\\ S_1x\end{pmatrix}$ is also an eigenvector of $DS(x)$. This fact will be needed later in Theorem \ref{thm:non singularity}.
\end{remark}

\begin{remark}\label{rmk:admissible maps}\rm
Since $\Omega$ is simply connected, \eqref{condition one} implies the existence of a real-valued $C^2$ function $s$ such that $Ds=S$. Replacing in \eqref{condition two} yields $Ds\times (Ds D^2s)=0$ and so $s$ solves the following quasilinear PDE
$$\left(s_{x_1}^2-s_{x_2}^2\right)s_{x_1x_2}+s_{x_1}s_{x_2}\left(s_{x_2x_2}-s_{x_1x_1}\right)=0,$$
which using Cauchy Kowalevski Theorem, \cite{bers1964partial}, has a unique local solution for a large class of initial data.

We list interesting admissible maps where conditions \eqref{condition one} and \eqref{condition two} can be easily verified.
\begin{itemize}
\item $Tx=(1+\alpha)x$ with $\alpha\neq -1$. In this case, $Sx=\alpha x= D\left(\frac{\alpha}{2}|x|^2\right)$. These maps $T$ correspond to dilations when $\alpha>0$, and to contractions when $-1<\alpha<0$. The case when $\alpha=-1$ is avoided since in that case, we get $T=0$ which is not injective. 
\item $Tx=(h(x_1),x_2)$, with $h$ a $C^1$ injective one variable real function. In this case, $Sx=(h(x_1)-x_1,0)$. These maps $T$ correspond to a transformation only in the horizontal variable. Similarly, transformations in the vertical direction $Tx=(x_1,h(x_2))$ are admissible maps.
\item Let $s$ be a $C^2$ function satisfying the Eikonal equation $|D s|=C$ for some $C>0$, and $S=Ds$. $S$ is the gradient of a function and with constant length then it clearly verifies \eqref{condition one} and \eqref{condition two}. Assume moreover that $DSx_0+I\neq 0$ at some $x_0$, then by the inverse function theorem $Tx=Sx+x$ is injective in a neighborhood of $x_0$ and is hence an admissible map in that neighborhood.
\end{itemize}

An interesting connection can be noticed between the admissible maps $S$ and the infinity Laplacian operator \cite{EvansInfinityLaplacian}. In fact if $S=Ds$ for $s$ a $C^2$ function and $S$ satisfies \eqref{condition two} then there exists a function $t(x)$ such that $(D^2s(x))(Ds(x))^t=t(x) (Ds(x))^t$ and hence dotting both sides with $(Ds)^t$, we conclude that $s$ satisfies the following inhomogeneous infinity Laplacian equation
$$\dfrac{1}{|Ds|^2} \langle Ds\, D^2s,Ds\rangle =t,$$
which is studied in \cite{Aronson1}, and \cite{LIU20125693}.
\end{remark}

\subsection{Existence of $\phi$}\label{sec: existence of phi}
Having found in Section 
\ref{sec: existence of rho} 
conditions on the map $T$ so 
that there exists a $C^2$ surface $\sigma_1$ that refracts vertical rays emitted from  $(x,0)$, with $x$ in a neighborhood of a point $x_0$  in $\Omega$, into the point $(Tx,a)$, we next use the analysis 
of Section \ref{sec: uniform 
refraction} to study the 
existence of a phase discontinuity $\phi$ defined in a neighborhood of every point $(Tx,a)$ so that the rays leave the lens $(\sigma_1,(\sigma_2,\phi))$ vertically into the point $(Tx,c)$, refer to Figure \ref{Imaging Problem} and to Section \ref{subsec: setup imaging}. 

More specifically, from Theorem \ref{thm:sufficient for vertical field}, it is sufficient to find conditions on the map $S$ so that the solution $\rho$ to \eqref{eq:PDE of rho} is $C^3$, and the matrices $D^2\rho$ and $\mathcal A$ given in \eqref{eq:mathcal A}
are invertible at $x_0$.

%Assume that the map $T:\Omega\mapsto \Omega^*$ is $C^2$ and $Sx=Tx-x$ satisfies \eqref{condition one} and $\eqref{condition two}$ in a neighborhood of $x_0$. Let $z_0$ be such that $z_0>\dfrac{|Sx_0|}{\sqrt{\kappa_1^2-1}},$ then from Corollary \eqref{cor:result} there exists $\rho$ a unique $C^2$ solution to \eqref{eq:PDE of rho} in a neighborhood of $x_0$ such that $a-\rho(x_0)=z_0$. Since $S$ is assumed to be $C^2$ then from \eqref{eq:PDE of rho} $\rho$ is $C^3$.
%We separate the cases when $x_0$ is a fixed point of $T$, i.e. $Sx_0=0$, and when it is not. In the latter case, from Remark \ref{rmk:class maps}, $S^t=\begin{pmatrix} S_1\\ S_2\end{pmatrix}$ and $S^t_{\perp}=\begin{pmatrix} -S_2\\ S_1\end{pmatrix}$ are eigenvectors of $DS$ at $x_0$.

Before stating the main result of this section, we need the following Lemma.
\begin{lemma}\label{lem:calculation}
Let $\varphi$ be the function given in \eqref{eq:psi}, then for every ${\bf y}\in \mathbb R^n$ such that $|{\bf y}|<\sqrt{\kappa_1^2-1}$ 
$$I+\dfrac{\kappa_1^2-1}{\kappa_1^2} \varphi^2(|{\bf y}|^2) ({\bf y}\otimes {\bf y})=\left(I+\varphi(|{\bf y}|^2)({\bf y}\otimes {\bf y})\right)\left(I+2\dfrac{\varphi'(|{\bf y}|^2)}{\varphi(|{\bf y}|^2)}{\bf y}\otimes {\bf y}\right)$$
\end{lemma}
\begin{proof}
\eqref{eq:psi}, \eqref{eq:phi'}, and the fact that $({\bf y}\otimes {\bf y})^2=|{\bf y}|^2({\bf y}\otimes {\bf y})$, yields the following
    \begin{align*}
\mathcal D&:=I+\dfrac{\kappa_1^2-1}{\kappa_1^2} \varphi^2(|{\bf y}|^2) ({\bf y}\otimes {\bf y})-\left(I+\varphi(|{\bf y}|^2)({\bf y}\otimes {\bf y})\right)\left(I+2\dfrac{\varphi'(|{\bf y}|^2)}{\varphi(|{\bf y}|^2)}{\bf y}\otimes {\bf y}\right)\\
&=({\bf y}\otimes {\bf y})\varphi(|{\bf y}|^2)\left(\dfrac{\kappa_1^2-1}{\kappa_1^2}\varphi(|{\bf y}|^2)-\dfrac{1}{\kappa_1\sqrt{|{\bf y}|^2+1}}-1-\dfrac{\varphi(|{\bf y}|^2)}{\kappa_1\sqrt{|{\bf y}|^2+1}}|{\bf y}|^2\right)\\
&=\dfrac{({\bf y}\otimes {\bf y})\,\varphi(|{\bf y}|^2)}{\kappa_1^2\sqrt{|{\bf y}|^2+1}}\left(\varphi(|{\bf y}|^2)\left((\kappa_1^2-1)\sqrt{|{\bf y}|^2+1}-\kappa_1|{\bf y}|^2\right)-\kappa_1\left(1+\kappa_1\sqrt{|{\bf y}|^2+1}\right)\right)
    \end{align*}
Noticing that for every $\tau\in [0,\kappa_1^2-1)$ 
$$(\kappa_1^2-1)\sqrt{\tau+1}-\kappa_1\tau=\left(\kappa_1-\sqrt{\tau+1}\right)\left(\kappa_1\sqrt{\tau+1}+1\right)=\dfrac{\kappa_1}{\varphi(\tau)}\left(\kappa_1\sqrt{\tau+1}+1\right),$$
we conclude that $\mathcal D=0$.
\end{proof}

\begin{theorem}\label{thm:non singularity}
Given $x_0\in \Omega$, $T$ a $C^1$ a map, with $Sx=Tx-x$ satisfying \eqref{condition one}, \eqref{condition two} in a neighborhood of $x_0$, $a>0$ verifying \eqref{eq:thickness} at $x_0$, let $\rho$ be a positive $C^2$ solution to \eqref{eq:form 2 of PDE} with $a-\rho(x_0)=z_0$ such that $z_0$ satisfies \eqref{eq:inequality local}, then at $x=x_0$
\begin{align}
D^2\rho&=\dfrac{\varphi}{z_0}\left(I+\dfrac{2}{z_0^2}\dfrac{\varphi'}{\varphi}(S\otimes S)\right)\left(\dfrac{\varphi}{z_0^2}(S\otimes S)-DS\right) \label{eq:hessian at x0}\\
\mathcal A&=\left(I+\dfrac{2}{z_0^2}\dfrac{\varphi'}{\varphi}(S\otimes S)\right)(I+DS)\label{eq:A at x0}
\end{align}
with  $\varphi$ and $\varphi'$ given in \eqref{eq:psi} and \eqref{eq:phi'} are evaluated at $\left|\dfrac{Sx_0}{z_0}\right|^2$, $S$ and $DS$ are evaluated at $x_0$.

Therefore, $\mathcal A$ is invertible at $x_0$ if and only if $DT=I+DS$ is invertible i.e. $T$ is a diffeomorphism in a neighborhood of $x_0$. Moreover,
\begin{enumerate}
\item If $Sx_0=(0,0)$, i.e. $x_0$ is a fixed point of $T$, then $D^2\rho(x_0)$ is invertible if and only if $DS(x_0)$ is invertible.
\item If $Sx_0\neq (0,0)$ then denoting by $\zeta$ and $\zeta_{\perp}$ the eigenvalues of $DS$ corresponding to the eigenvectors $S^t$ and $S^t_{\perp}$, we get that $D^2\rho(x_0)$ is invertible if and only if $\zeta\neq \dfrac{|Sx_0|^2}{z_0^2}\varphi\left(\left|\dfrac{Sx_0}{z_0}\right|^2\right)$ and $\zeta_{\perp}\neq 0$.%; and $\mathcal A(x_0)$ is invertible if and only if $\zeta,\zeta_{\perp}\neq -1$.
\end{enumerate}
\end{theorem}

\begin{proof}
Recall from \eqref{eq:form 2 of PDE} that
$\rho_{x_i}(x)=-V_i(x,a-\rho(x)).$
Letting $z(x)=a-\rho(x)$, then  \eqref{eq:form 2 of PDE}, \eqref{eq:vixj},\eqref{eq:Viz}, and the formula for $V$ yield
\begin{align*}
\rho_{x_ix_j}&=-\dfrac{\partial V_i}{\partial x_j}-V_j \dfrac{\partial V_i}{\partial z}=-2\dfrac{S_i}{z^3}(S\cdot S_{x_j})\varphi'%\left(\left|\dfrac{Sx}{z}\right|^2\right)
-\dfrac{1}{z}\varphi%\left(\left|\dfrac{Sx}{z}\right|^2\right)
(S_i)_{x_j}-\varphi%\left(\left|\dfrac{Sx}{z}\right|^2\right)
\dfrac{S_j}{z}\left(-\dfrac{2|S|^2}{z^4}\varphi'%\left(\left|\dfrac{Sx}{z}\right|^2\right)
S_i-\dfrac{1}{z^2}\varphi%\left(\left|\dfrac{Sx}{z}\right|^2\right)
S_i\right),
\end{align*}
and so
$$D^2\rho(x_0)=-\dfrac{2}{z_0^3}\varphi'(S\otimes S) DS-\dfrac{1}{z_0}\varphi DS+\left(\dfrac{2|S|^2}{z_0^5}\varphi\varphi'+\dfrac{1}{z_0^3}\varphi^2\right)(S\otimes S),
$$
with $S, DS$ evaluated at $x_0$, and $\varphi,\varphi'$ at $\left|\dfrac{Sx_0}{z_0}\right|^2$.

Since $(S\otimes S)^2=|S|^2(S\otimes S)$, then simplifying the above expression
\begin{align*}
D^2\rho(x_0)&=-\dfrac{\varphi}{z_0}\left(I+\dfrac{2}{z_0^2}\dfrac{\varphi'}{\varphi}S\otimes S\right)DS+\dfrac{2(S\otimes S)^2}{z_0^5}\varphi\varphi'+\dfrac{1}{z_0^3}\varphi^2(S\otimes S)\\
&=-\dfrac{\varphi}{z_0}\left(I+\dfrac{2}{z_0^2}\dfrac{\varphi'}{\varphi} S\otimes S\right)DS+\dfrac{\varphi^2}{z_0^3}\left( I+\dfrac{2}{z_0^2}\dfrac{\varphi'}{\varphi}S\otimes S\right)S\otimes S\\
&=\dfrac{\varphi}{z_0}\left( I+\dfrac{2}{z_0^2}\dfrac{\varphi'}{\varphi} S\otimes S\right)\left(\dfrac{\varphi}{z_0^2}S\otimes S-DS\right),
\end{align*}
hence obtaining \eqref{eq:hessian at x0}.
Since $\varphi,\varphi'>0$, and $S\otimes S$ is positive semidefinite, we deduce that $D^2\rho(x_0)$ is invertible if and only if $\dfrac{\varphi}{z_0^2}S\otimes S-DS$ is invertible. We consider the two cases:
\begin{itemize}
\item If $Sx_0=(0,0)$ then $\dfrac{\varphi}{z_0^2}S\otimes S-DS=-DS$, and so $D^2\rho(x_0)$ is invertible if and only if $DS(x_0)$ is invertible.
\item If $Sx_0\neq (0,0)$. We have that $S^tx_0$ and $S^t_{\perp}x_0$ are eigenvectors of $S\otimes S$ with corresponding eigenvalues $|Sx_0|^2$ and $0$. Also, from Remark \ref{rmk:admissible maps}, $S^tx_0$ and and $S^t_{\perp}x_0$ are eigenvectors of $DS(x_0)$ with corresponding eigenvalues denoted by $\zeta$ and $\zeta_{\perp}$. Hence $D^2\rho(x_0)$ is invertible if and only if $\varphi\left(\left|\dfrac{Sx_0}{z_0}\right|^2\right)\dfrac{|Sx_0|^2}{z_0^2}-\zeta\neq 0$ and $\zeta_{\perp}\neq 0$.
\end{itemize}

Next, we simplify the matrix $\mathcal A$ given in \eqref{eq:mathcal A}. From \eqref{eq:form 2 of PDE}
\begin{equation}\label{eq:DrhoDrho}
D\rho(x_0)\otimes D\rho(x_0)=\varphi^2\left(\left|\dfrac{Sx_0}{z_0}\right|^2\right)\dfrac{Sx_0\otimes Sx_0}{z_0^2}.
\end{equation}
From \eqref{eq:nice Delta} and \eqref{eq:psi}
\begin{align*}\kappa_1^2+\Delta(x_0)&=\kappa_1^2+\dfrac{\kappa_1^2|(Sx_0,z_0)|-\kappa_1z_0}{\kappa_1z_0-|(Sx_0,z_0)|}=\dfrac{\kappa_1^3-\kappa_1}{\kappa_1-\sqrt{\left|\dfrac{Sx_0}{z_0}\right|^2+1}}=(\kappa_1^2-1)\varphi\left(\left|\dfrac{Sx_0}{z_0}\right|^2\right),
\end{align*}
and so 
\begin{equation}\label{eq:factor of the Hessian}
    \dfrac{(a-\rho(x_0))(1-\kappa_1^2)}{\kappa_1^2+\Delta(x_0)}=-\dfrac{z_0}{\varphi\left(\left|\dfrac{Sx_0}{z_0}\right|^2\right)}.
\end{equation}
Replacing \eqref{eq:hessian at x0}, \eqref{eq:DrhoDrho}, and \eqref{eq:factor of the Hessian} in the expression of $\mathcal A$ in \eqref{eq:mathcal A}, we get
$$\mathcal A=I+\dfrac{\kappa_1^2-1}{\kappa_1^2}\dfrac{\varphi^2}{z_0^2}S\otimes S-\dfrac{z_0}{\varphi}D^2\rho.$$
Hence from Lemma \ref{lem:calculation} applied to ${\bf y}=\dfrac{Sx_0}{z_0}$ and \eqref{eq:hessian at x0} 
\begin{align*}
    \mathcal A&=\left(I+\dfrac{\varphi}{z_0^2}S\otimes S\right)\left(I+\dfrac{2}{z_0^2}\dfrac{\varphi'}{\varphi}S\otimes S\right)-\left( I+\dfrac{2}{z_0^2}\dfrac{\varphi'}{\varphi} S\otimes S\right)\left(\dfrac{\varphi}{z_0^2}S\otimes S-DS\right)
    &=\left(I+\dfrac{2}{z_0^2}\dfrac{\varphi'}{\varphi}S\otimes S\right)(I+DS).
\end{align*}
Hence, since $\varphi,\varphi'>0$ and $S\otimes S$ is positive semidefinite, we conclude that $\mathcal A(x_0)$ is invertible if and only if $I+DS(x_0)$ is invertible.
\end{proof}

Theorem \ref{thm:non singularity} concludes our analysis of the imaging problem in Section \ref{subsec: setup imaging} which can be summarized as follows. 
Given $\Omega$ open, and simply connected domain, $a>0$, $T:\Omega\mapsto \Omega^*$ a $C^2$ bijective map, define $Sx=Tx-x$. Let $x_0\in \Omega$, and assume $a$ and $Sx_0$ are such that $a>\dfrac{|Sx_0|}{\sqrt{\kappa_1^2-1}}$. Assume moreover $S$ satisfies \eqref{condition one} and \eqref{condition two} in a neighborhood of $x_0$.
Then, from Corollary \ref{cor:result}, for every $z_0$ satisfying inequality \eqref{eq:inequality local}, there exists $\rho$ positive solution to \eqref{eq:PDE of rho} such that $a-\rho(x_0)=z_0$. Since $S$ is $C^2$ then the solution $\rho$ is $C^3$. Letting $\sigma_1=\{(x,\rho(x))\}$ separating media I and II, $\sigma_1$ refracts the vertical rays emitted from $(x,0)$, $x$ in a neighborhood of $x_0$, into the $(Tx,a)$. Letting $c>a$, if moreover, $T$ is a diffeomorphism in a neighborhood of $x_0$, and $S$ satisfies cases (1), or (2) of Theorem \ref{thm:non singularity} then, from Theorem \ref{thm:sufficient for vertical field}, there exists a $C^1$ phase discontinuity $\phi$ defined in a neighborhood of every point $(Tx,a)\in \sigma_2,$ $x$ in a neighborhood of $x_0$, such that the ray at $(Tx,a)$ is refracted by the metasurface $(\sigma_2,\phi)$ into the point $(Tx,c).$
\subsection{ Examples}\label{subsec:Examples}
This section elaborates on the examples of admissible maps $T$ mentioned in Remark \ref{rmk:admissible maps}, and examines whether they satisfy the conditions of Theorem \ref{thm:non singularity}.

\paragraph{\bf Example 1.} Consider the maps $Tx=(1+\alpha)x$  in a neighborhood of $x_0=(0,0)$ with $\alpha\neq -1$. 

$T$ is clearly a diffeomorphism and so from Theorem \ref{thm:non singularity} the corresponding matrix $\mathcal A(x_0)$ is invertible. Also, $Sx_0=\alpha x_0=0$ and $DS(x_0)=\alpha I$. So from Theorem \ref{thm:non singularity} (case (1)) if $\alpha\neq 0$ then $D^2\rho$ is invertible, and hence for every $a>z_0>0$ the imaging problem has a local solution in a neighborhood of $x_0=(0,0)$ with $a-\rho(x_0)=z_0$.

In the case of $\alpha=0$, we cannot apply Theorem \ref{thm:sufficient for vertical field} since $DS=0$ and hence $D^2\rho$ is not invertible, however, notice that in this case $T$ is the identity map, then letting $\sigma_1$ be a horizontal plane below the plane $\{x_3=a\}$,  i.e. $\rho$ is constant, and $\varphi=0$, then the incident vertical rays are not deviated by the lens $(\sigma_1,(\sigma_2,\phi))$ and so the map $T$ can still be achieved in this case. 

%See Figure \ref{fig: graphs} for the graphs of the solutions $\rho$ in the two dimensional case for several values of $\alpha.$

\paragraph{\bf Example 2.} Assume $T$ only changes one of the coordinates, say for example $Tx=(h(x_1),x_2)$ with $h\in C^2$ and injective, and let $x_0=(0,0).$ Notice that in this case $Sx=(h(x_1)-x_1,0)$, $DS(x)=\left(\begin{matrix} h'(x_1)-1 & 0 \\ 0 & 0\end{matrix}\right)$. 

If $h(0)=0$, then $Sx_0=(0,0)$ but since $DS(x_0)$ is singular so from Theorem \ref{thm:non singularity} (case (1)) $D^2\rho$ is singular at $x_0$ and the existence of phase discontinuity $\phi$, in this case, cannot be deduced from Theorem \ref{thm:sufficient for vertical field}.

If $h(0)\neq 0$, then $Sx_0\neq (0,0)$. In this case $\zeta_{\perp}=0$ is the eigenvalue of $DS(x_0)$ corresponding to $S_{\perp}^tx_0$ and hence again by Theorem \ref{thm:non singularity} (case (2)) $D^2\rho(x_0)$ is singular and the existence of phase discontinuity $\phi$ cannot be deduced from Theorem \ref{thm:sufficient for vertical field}.

\paragraph{\bf Example 3.} Consider a function $s$ on $\Omega$ solution to the Eikonal equation $|D s|=C$ for some $C>0$, and $S=Ds$.
Let $x_0=(0,0)$, notice that since $|Sx_0|=C>0$ then $S^tx_0\neq (0,0)$. Squaring then differentiating the Eikonal equation we get that $D^2s(x_0)=DS(x_0)$ has an eigenvalue equal to $0$ corresponding to the eigenvector $S^tx_0$. From Theorem \ref{thm:non singularity}, $D^2\rho(x_0)$ and $\mathcal A(x_0)$ are invertible if and only if the eigenvalue $\zeta_{\perp}$ of $DS$ corresponding to $S_{\perp}^t$ is different than $0$ and $-1$. 

For example, let $s(x)=|x-\gamma|,$ with $\gamma=(\gamma_1,\gamma_2)\neq (0,0)$ solving $|D s|=1$. In this case, $Sx=\dfrac{x-\gamma}{|x-\gamma|}$ and 
$$DS=D^2s=\dfrac{1}{|x-\gamma|}I-\dfrac{(x-\gamma)\otimes (x-\gamma)}{|x-\gamma|^3}.$$
At $x_0=(0,0)$, the eigenvalues of $DS$ are $\zeta=0$ corresponding to $S^t(x_0)$ and $\zeta_{\perp}=\dfrac{1}{|\gamma|}$ corresponding to $S_{\perp}^t(x_0)$, hence from Theorem \ref{thm:non singularity} the matrix $D^2\rho$ and $\mathcal A$ are invertible at $x_0$ and a solution to the imaging problem exists in a neighborhood of $(0,0)$.

%Since $\varphi$ and $\varphi'$ are positive, and $S\otimes S$ is 
%a positive semi-definite matrix then  
%$$\det\left(\varphi I+\dfrac{2}{z_0^2}\varphi' S\otimes S\right)>0,$$ concluding that $D^2\rho(x_0)$ is non-singular if and only if \eqref{eq: non singular rho}.

\subsection{The two-dimensional case}\label{subsec:2D} We end the paper by providing an easier presentation of the existence results in the two-dimensional case. Here, $\Omega\subseteq \mathbb R$, $\sigma_1=\{(t,\rho(t)\}_{t\in \Omega}$ with $\rho$ to be calculated, and $\sigma_2$ is contained in the horizontal line $y=a$ with $a>0$ and phase discontinuity $\phi:=\phi(u_1,u_2)$ to be found. For simplicity, using an appropriate system of coordinates, we solve in a neighborhood of $t_0=0$.

Given a $C^2$ bijective map $T$ defined on the interval $(-\tau,\tau)$, and letting $S=Tx-x$, from Section \ref{sec: existence of rho}, assume $a>\dfrac{|S(0)|}{\sqrt{\kappa_1^2-1}}$, and $a>z_0>\dfrac{|S(0)|}{\sqrt{\kappa_1^2-1}}$ we are interested in finding positive solution to the IVP
\begin{equation}\label{eq:ODE}\begin{cases}\rho'(t)=\dfrac{\kappa_1S(t)}{\sqrt{|S(t)|^2+(a-\rho(t))^2}-\kappa_1(a-\rho(t))}\\
a-\rho(0)=z_0\end{cases}.\end{equation}
By Picard's Theorem such a system has a unique local solution for every $C^2$ map $S$.
\begin{figure}[h]
%\centering
\includegraphics[scale=0.3]{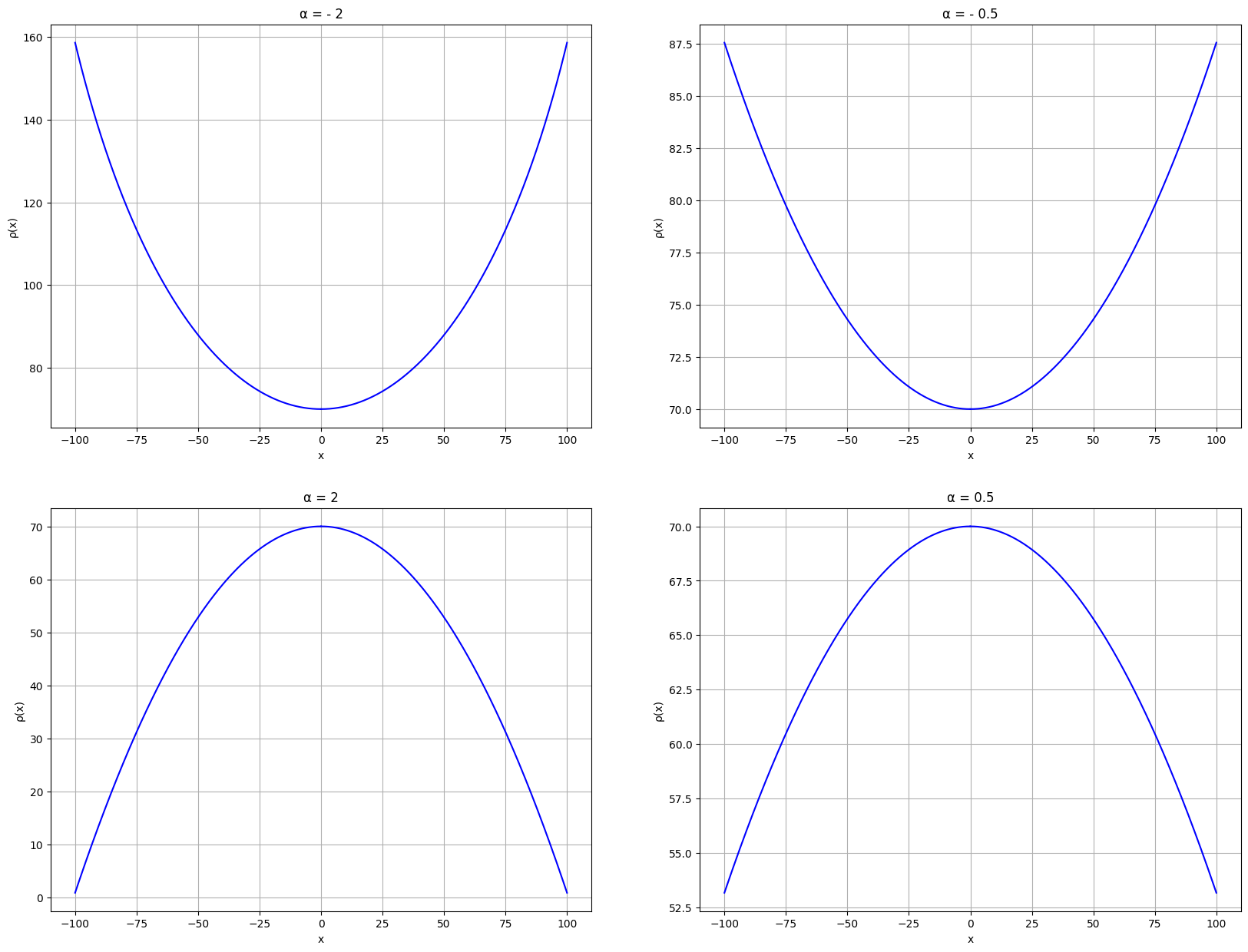}
\caption{}
\label{fig: graphs}
\end{figure}
In the particular case when $T(t)=(1+\alpha)t$, we have $S(t)=\alpha t$, and using the notation $\tilde \rho=a-\rho$ \eqref{eq:ODE} can be written as follows:

$$
\begin{cases}
    \dfrac{d\tilde \rho}{dt}=\dfrac{\dfrac{\alpha t}{\tilde \rho(t)}}{\kappa_1-\sqrt{\dfrac{\alpha^2\,t^2}{(\tilde \rho(t))^2}+1}}\\
    \tilde \rho(0)=z_0
\end{cases}.$$
The ODE in this case is homogeneous of the form $\tilde \rho'=F\left(\frac{t}{\tilde \rho}\right)$ that can be solved explicitly using the substitution $v(t)=\dfrac{t}{\tilde \rho(t)}$, See Figure \ref{fig: graphs} for the graph of solutions corresponding to different values of  $\alpha,$ and $z_0=70$.

In the more general case, given a map $T$, and $\rho$ a 
positive local solution to \eqref{eq:ODE} satisfying \eqref{eq:imaging inequality}. The curve 
$\sigma_1$ refracts rays emitted from 
$(t,0)$ with $t$ in a neighborhood to 
$t_0=0$ into the point $(Tx,a)$. For 
the existence of a phase discontinuity in a 
neighborhood of every point in $\sigma_2$ so that rays are refracted by $\sigma_2$ with vertical direction,  we need, from Theorems \ref{thm:sufficient for vertical field}, and \ref{thm:non singularity}, that
$S'(0)\neq \varphi\left(\dfrac{|S(0)|^2}{z_0^2}\right)\dfrac{S(0)^2}{z_0^2}$, and $S'(0)\neq -1$.

{\footnotesize
\bibliography{references}
}
%\bibliography{references}
\bibliographystyle{IEEEtranS}
\end{document}